\documentclass[10pt,twoside,a4paper]{article}

\usepackage[utf8]{inputenc}

\usepackage[T1]{fontenc}

\usepackage{amsmath, bm}

\usepackage{amssymb}

\usepackage{microtype}
\usepackage{lmodern}
\usepackage{mathtools}
\usepackage[pagestyles]{titlesec}

\usepackage{placeins}
\usepackage{float}

\usepackage[backend=biber,maxnames=10,backref=true,hyperref=true,safeinputenc]{biblatex}

\DefineBibliographyStrings{english}{%
	backrefpage = {cited on page},
	backrefpages = {cited on pages},
}

\DeclareFieldFormat[report]{title}{``#1''}
\DeclareFieldFormat[book]{title}{``#1''}
\AtEveryBibitem{\clearfield{url}}
\AtEveryBibitem{\clearfield{note}}

\usepackage{graphicx}
\graphicspath{{images/}}

\usepackage{rotating}
\usepackage{enumerate}
\usepackage{multirow}
\usepackage{booktabs}

\usepackage{tikz,pgfplots}
\usepackage{tikz-3dplot}
\usetikzlibrary{arrows,shapes,fadings,intersections,decorations.pathreplacing,decorations.markings,positioning,patterns}
\usetikzlibrary{quotes}
\usetikzlibrary{angles}
\usetikzlibrary{babel}
\usetikzlibrary{calc, arrows.meta}
\usetikzlibrary{shapes.geometric}
\pgfplotsset{compat=newest}

\title{Invertibility of the Fourier Diffraction Relation\\ in Raster Scan Diffraction Tomography}

\author{Peter Elbau$^{1}$\\
	{\footnotesize\href{mailto:email}{peter.elbau@univie.ac.at}}
	\and
	Noemi Naujoks$^{1,2}$\\
	{\footnotesize\href{mailto:email}{noemi.naujoks@univie.ac.at}}
}

\date{\today}

\usepackage[pdftex,colorlinks=true,linkcolor=blue,citecolor=green!50!black,urlcolor=blue,bookmarks=true,bookmarksnumbered=true,pdfusetitle]{hyperref}
\hypersetup
{
	pdfauthor={P. Elbau, N. Naujoks},
	pdftitle={Invertibility of the Fourier Diffraction Relation in Raster Scan Diffraction Tomography},
}

\usepackage[hyperref,amsmath,thmmarks]{ntheorem}
\usepackage{aliascnt}

\newtheorem{lemma}{Lemma}[section]

\newaliascnt{proposition}{lemma}
\newtheorem{proposition}[proposition]{Proposition}
\aliascntresetthe{proposition}

\newaliascnt{corollary}{lemma}
\newtheorem{corollary}[corollary]{Corollary}
\aliascntresetthe{corollary}

\newaliascnt{theorem}{lemma}
\newtheorem{theorem}[theorem]{Theorem}
\aliascntresetthe{theorem}

\theorembodyfont{\normalfont}
\newaliascnt{definition}{lemma}
\newtheorem{definition}[definition]{Definition}
\aliascntresetthe{definition}

\newaliascnt{problem}{lemma}
\newtheorem{problem}[problem]{Problem}
\aliascntresetthe{problem}

\newaliascnt{remark}{lemma}
\newtheorem{remark}[remark]{Remark}
\aliascntresetthe{remark}

\newaliascnt{example}{lemma}

\aliascntresetthe{example}

\theoremstyle{nonumberplain}
\theoremseparator{:}
\theoremheaderfont{\normalfont\itshape}

\theoremsymbol{\ensuremath{\square}}
\newtheorem{proof}{Proof}

\usepackage[a4paper,centering,bindingoffset=0cm,marginpar=2cm,margin=2.5cm]{geometry}
\usepackage[pagestyles]{titlesec}

\titleformat{\section}[block]{\large\sc\filcenter}{\thesection.}{0.5ex}{}[]
\titleformat{\subsection}[runin]{\bf}{\thesubsection.}{0.5ex}{}[.]

\newpagestyle{headers}

{
	
	\headrule
	
	\sethead[\footnotesize\thepage][\footnotesize\sc  P. Elbau, N. Naujoks][]{}{\footnotesize\sc Invertibility of the Fourier Diffraction Relation in Raster Scan Diffraction Tomography}{\footnotesize\thepage}
	
	\setfoot{}{}{}
	
}

\usepackage[font=footnotesize,format=plain,labelfont=sc,textfont=sl,width=0.75\textwidth,labelsep=period]{caption}
\usepackage[labelfont = normalfont]{subcaption}

\usepackage{dsfont}
\usepackage{bbm}
\newcommand{\N}{\mathds{N}}

\newcommand{\R}{\mathds{R}}
\newcommand{\C}{\mathds{C}}
\newcommand{\Sp}{\mathbb{S}}

\newcommand{\ui}{u^{\text{inc}}}

\renewcommand{\S}{\mathds S}

\let\RE\Re
\let\Re=\undefined
\DeclareMathOperator{\Re}{\RE e}

\let\IM\Im
\let\Im=\undefined
\DeclareMathOperator{\Im}{\IM m}
\DeclareMathOperator{\supp}{supp}

\newcommand{\norm}[1]{\left\|#1\right\|}

\newcommand{\inner}[2]{\left<#1,#2\right>}

\begin{document}
	
	\maketitle

	\thispagestyle{empty}
	\begin{center}
		\hspace*{5em}
		\parbox[t]{12em}{\footnotesize
			\hspace*{-1ex}$^1$Faculty of Mathematics\\
			University of Vienna\\
			Oskar-Morgenstern-Platz 1\\
			A-1090 Vienna, Austria}
		\hfil
		\vspace*{0.5cm}
		\hspace*{5em}
		\parbox[t]{18em}{\footnotesize
			\hspace*{-1ex}$^2$Christian Doppler Laboratory for Mathematical\\
			\hspace*{1em}Modelling and Simulation of Next Generation\\
			\hspace*{1em}Medical Ultrasound Devices (MaMSi)\\
			Oskar-Morgenstern-Platz 1\\
			A-1090 Vienna, Austria}
	\end{center}

	\begin{abstract}

		Diffraction tomography aims to recover an object's scattering potential from measured wave fields. In the classical setting, the object is illuminated by plane waves from many directions, and the Fourier diffraction theorem provides a direct relation between the Fourier transform of the object's scattering potential and the Fourier transform of the measurements.

		In many practical imaging systems, however, focused beams are used instead of plane waves. These beams are then translated across the object to bring different regions of interest into focus.

		This article discusses what information about the scattering potential can be extracted from such measurements. As in the classical case, the analysis is based on a recently derived Fourier diffraction relation that relates the measurements to the Fourier coefficients of the scattering potential. However, this relation does not immediately provide an explicit reconstruction formula, but instead leads to a linear equation system for the Fourier coefficients.

		We therefore prove in this work that all Fourier coefficients appearing in these relations are in dimensions higher than two generically uniquely determined. In the two-dimensional case, on the other hand, only a particular subset of the Fourier coverage is uniquely recoverable, while on the remaining region distinct coefficients may produce identical data.

	\end{abstract}
	
	\section{Introduction}
Diffraction tomography is a widely used inverse scattering technique that aims to reconstruct the spatial distribution of an object’s scattering potential from measurements of scattered wave fields. A prominent application is ultrasound tomography \cite{Bor92, DurLitBabChaAze05b, SimHuaDur08}, where both amplitude and phase information of the scattered fields are available.

Diffraction tomography relies on the Born approximation, under which the dependence of the scattered field on the scattering potential becomes linear and the corresponding Fourier diffraction theorem \cite{Wol69} provides an explicit representation of the scattering potential in terms of the measurements in Fourier space. More precisely, it states that the Fourier-transformed measurements coincide with the Fourier transform of the scattering potential along a semicircle in two dimensions or a hemisphere in higher dimensions. Incorporating a collection of measurements from multiple views then generates data on a union of distinct hemispheres in Fourier space, enabling reconstruction of the scattering potential through \textit{filtered backpropagation} \cite{Dev82, KakSla01, KirQueRitSchSet21, Wol69}.

However, the classical theory relies on idealized assumptions about both the incident field and the measurement geometry. In particular, the object is assumed to be illuminated by monochromatic plane waves arriving from a broad range of directions \cite{KakSla01, Nat15, Wol69}.
In practice, though, many imaging systems, such as medical ultrasound \cite{Hos19}, use focused beams that concentrate the energy at a focal region to improve the spatial resolution.
Rather than illuminating the entire object from different directions, these beams are translated to bring different regions into focus.
This scanning setup differs fundamentally from the plane-wave, full-angle illumination assumed in the conventional framework, highlighting a gap between theory and practical setups.

To address this discrepancy, recent work \cite{KirNauSchYan24} extended the classical diffraction tomography framework to account for focused beams by modeling them as superpositions of plane waves while still assuming full-angle data. Building on this, \cite{ElbNauSch26_preprint} further extended diffraction tomography to a raster scan geometry in which a focused beam is emitted from only one side and then translated across the object along a fixed hyperplane. Within this framework, a new Fourier diffraction relation was derived, expressing the Fourier-transformed measurements in terms of the Fourier coefficients of the scattering potential of the object. Here, depending on the scan configuration, some measured data correspond directly to individual coefficients, while others are linear combinations of two coefficients.

 Building on this framework, the present article focuses on the mathematical structure of the resulting Fourier diffraction relation. More precisely, we investigate whether, and under which conditions, the Fourier coefficients of the scattering potential can be uniquely recovered from the resulting system of linear equations.

While the physically relevant case is clearly the three-dimensional setting, we discuss the problem in arbitrary dimensions, with particular attention to the theoretically interesting two-dimensional case where the dimensions of the gathered data and the Fourier data of the scattering potential are equal.

Our results show that in dimensions higher than two, all Fourier coefficients appearing in the equations are, at least in the generic case, uniquely recoverable from the measured data. In two dimensions, however, only a subset of the coefficients is uniquely determined, while on the remaining region, distinct values for the Fourier coefficients produce identical measurements.

These findings are essential for the reconstruction process of raster scan diffraction tomography: only those Fourier coefficients that are uniquely determined by the measurement data can be incorporated into backpropagation. The established invertibility conditions therefore provide the mathematical foundation for guaranteeing unique reconstructibility of the scattering potential within the accessible frequency coverage.

\subsection*{Outline}
We begin in \autoref{sec:problem} with the formulation of the inverse problem and recall the Fourier diffraction relation obtained in \cite{ElbNauSch26_preprint}. In \autoref{sec:couplingset}, we then investigate which pairs of Fourier coefficients of the scattering potential are interconnected by these relations. This structural analysis already reveals fundamental differences between the two-dimensional, the three-dimensional, and the higher-dimensional cases. Accordingly, we address these situations separately in \autoref{sec:higher}, \autoref{subsec:3d}, and \autoref{sec:two}, where we start with the problem in more than three dimensions as the surplus of measurement dimensions simplifies the analysis, and we end with the critical two-dimensional case.

\begin{figure}[t]
	\centering
		\includegraphics[scale = 1]{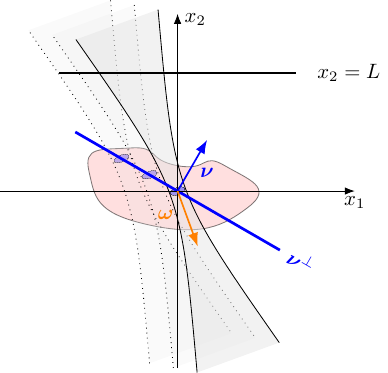}
	\caption{ Schematic overview of the general scan geometry in two dimensions. A focused incident beam propagates in a direction $\omega\in\Sp^{1}$ and scans an object by moving the focal point along a line $\nu^\perp$ orthogonal to the direction $\nu\in\Sp^1$. The resulting scattered waves are measured at every point on a receiver line $\{x\in\R^2\mid x_2 = L\}$ outside the object.}
	\label{fig:scangeo}
\end{figure}

	\section{Problem formulation}\label{sec:problem}
We begin by briefly outlining the experimental setup underlying to our analysis, see \autoref{fig:scangeo}. The aim is to image an unknown object, represented by its complex-valued scattering potential $f\in L^1(\R^d)$ in the arbitrary spatial dimension $d\in\N\setminus\{1\}$, where we are, of course, mainly interested in the physically relevant cases $d\in\{2,3\}$. We assume that the object is embedded in a homogeneous background medium. More precisely, we assume that $\supp(f)\subseteq\mathcal B^d_r$, where we denote by
\[
\mathcal{B}_{r}^d\coloneqq \{ x\in\R^d\mid\norm{x} < r \}
\]
the open ball in $\R^d$ of radius $r>0$ centered at the origin.

As incident wave, we consider a focused acoustic beam which propagates in the direction $\omega\in\Sp^{d-1}\coloneqq\partial\mathcal B_1^d$ and which is then successively translated across the object along a predefined scan plane 
\[
\nu^\perp \coloneqq\left\lbrace x\in\R^d \mid \inner{x}{\nu} = 0\right\rbrace,
\]
where  $\nu\in \Sp^{d-1}$ defines the scan normal and $\inner{\cdot}{\cdot}$ is the Euclidean inner product. For each position $y\in\nu^\perp$, the incident beam $\ui_y$ satisfies the homogeneous Helmholtz equation 
\[
(\Delta+k_0^2)\ui_y (x)=0,\qquad x\in\R^d,
\] 
where $k_0>0$ is the wave number, and is modeled as a translated Herglotz wave \cite{ColKre19},
\begin{align*}
	\ui_{{y}}( x)\coloneqq \int_{\mathbb S^{d-1}_{k_0}}a(s)e^{i\inner{x}{s}}e^{-i\inner{ y }{s}}dS({s}),\quad  {x}\in\R^{d}, \  {y}\in \nu^\perp,
\end{align*}
with Herglotz density $a\in L^2(\Sp_{k_0}^{d-1})$. Here, $\mathbb{S}_{k_0}^{d-1} \coloneqq \partial\mathcal B_{k_0}^d$ denotes the sphere of radius $k_0$ in $\R^d$ and $L^2(\Sp_{k_0}^{d-1})$ is the Lebesgue space of all complex-valued square-integrable functions on this sphere.
To reflect the fact that $\ui$ should propagate in the direction $\omega$, we exclude all plane wave components moving in a direction $s\in\Sp_{k_0}^{d-1}$ with $\inner s\omega\le0$, that is, we enforce $a(s)=0$ for those directions.

As prime example for such an incident focused beam, we consider a Gaussian beam, which we obtain with a Herglotz density of the form of
\begin{align*}
	a(s) \coloneqq \begin{cases}
		e^{-A\norm{\tilde{s}}^2}, \quad &\inner{s}{\omega} >0,
		\\ 0, & \inner{s}{\omega}\leq 0, 
	\end{cases}
\end{align*}
where $\tilde{s} \coloneqq s-\inner{s}{\omega}\omega$ is the component of $s$ orthogonal to $\omega$. Motivated by this, we further assume for the later analysis that the Herglotz density $a\in L^2(\Sp_{k_0}^{d-1})$ fulfills that
\begin{equation}\label{eq:supp_a}
a(s)\neq 0\quad\text{if and only if}\quad \inner{s}{\omega}>0.
\end{equation}

For the wave propagation, we work within the first Born approximation \cite{Wol69}, assuming that the values of $f$ are sufficiently small such that multiple scattering effects can be neglected. For each scan position $y\in\nu^\perp$, the corresponding scattered field $u_y$ is then governed by the inhomogeneous Helmholtz equation \cite[Section~3.3]{NatWub01}
\begin{align*}
	(\Delta+k_0^2)u_y(x) = -f(x) \ui_y(x),\qquad x\in\R^d,
 \end{align*}
together with the Sommerfeld radiation condition \cite{Som12}, ensuring uniqueness of the solution.

For each scan position, measurements of the scattered field are collected on the $(d-1)$-dimensional hyperplane $\R^{d-1}\times\{L\}$ with $L>r$ chosen sufficiently large such that this measurement plane does not intersect the support of the scattering potential of the object.

 The inverse problem now consists in reconstructing the scattering potential $f$ from the measured data
 \begin{equation}\label{eq:MeaseurementData}
m\colon(\R^{d-1}\times\{L\})\times\nu^\perp\to\C,\quad m(x,y)\coloneqq u_y(x).
 \end{equation}
Thus, the acquisition process yields a $2(d-1)$-dimensional dataset. In particular, the dimension of the data exceeds that of the object for $d>2$.

\subsection{Fourier diffraction relation for scanning data}
To recover from the measurement data in \autoref{eq:MeaseurementData} the scattering potential, we perform a $2(d-1)$-dimensional Fourier transform of the data $m$.
The Fourier diffraction theorem adapted to scanning configurations, derived in \cite{ElbNauSch26_preprint}, then provides an algebraic relation between these Fourier-transformed measurements and the Fourier transform of the scattering potential. Before stating this result, we introduce some notations that will be used throughout this article. 
 
For a non-zero vector $v\in \R^d$, the associated open half-space is given by
\begin{equation}\label{def:H}
	\Omega_v\coloneqq\left\lbrace x\in\R^d\mid \inner{x}{v}>0\right\rbrace, 
\end{equation}
and the hemisphere in the direction of $v$ with radius $k_0$ is denoted by
\begin{equation}\label{def:Somega}
	S_v \coloneqq \Sp_{k_0}^{d-1} \cap \Omega_v.
\end{equation}

Moreover, for $v\in \mathbb{S}^{d-1}$, the reflection of a point $x$ across the hyperplane $v^\perp$ is given by the Householder transform $H_v\colon\R^d\to\R^d$,
\begin{equation}\label{eq:I}
	H_v x \coloneqq x - 2 \inner{x}{v} v.
\end{equation}
In particular, we consider the hemisphere $S_\omega$ oriented along the beam direction $\omega\in \Sp^{d-1}$, and the reflection $H_\nu$ across the scan plane $\nu^\perp$ orthogonal to the scan normal $\nu\in\Sp^{d-1}$. We then divide $S_\omega$ into two disjoint parts $S_\omega= \Sigma_1\sqcup\Sigma_2$, where
	\begin{align}\label{def:Sigma}
		\Sigma_1 \coloneqq \{ \sigma \in S_\omega \mid H_\nu\sigma \notin S_\omega \} \quad \text{and} \quad
		\Sigma_2 \coloneqq \{ \sigma \in S_\omega \mid H_\nu\sigma \in S_\omega \}.
	\end{align}
See \autoref{subfig:Sigmas} for an illustration of this decomposition.
\begin{figure}[t]
	\centering
		\begin{subfigure}{0.47\textwidth}
			\includegraphics{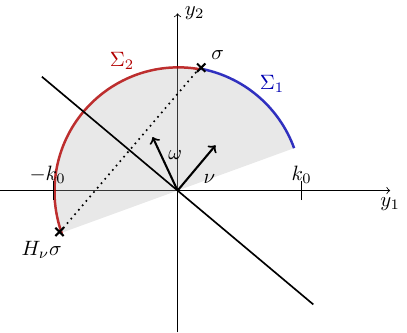}
			\caption{Symmetry on $S_\omega$ with respect to the scan plane $\nu^\perp$.}
			\label{subfig:Sigmas}
		\end{subfigure}\hfil
		\begin{subfigure}{0.45\textwidth}
			\includegraphics{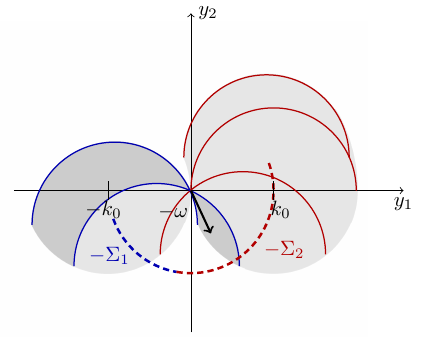}
			\caption{Union of upper semicircles centered at $-S_\omega$.}
			\label{subfig:union}
		\end{subfigure}
		
	\caption{(a) Illustration of the symmetry with respect to the scan plane $\nu^\perp$. The subset $\Sigma_1$ (blue) contains directions whose reflections lie outside $ S_{\omega}$, whereas for $\Sigma_2$  (red) the reflections remain inside. (b) Region in Fourier space that can be accessed by the measurements. It is obtained as the union of semicircles centered at $-S_\omega = -(\Sigma_1 \sqcup \Sigma_2)$. Representative semicircles are colored according to their centers: blue for $\Sigma_1$ and red for $\Sigma_2$. The dark gray region corresponds to $\mathcal{Y}_1$ (centers in $-\Sigma_1$), and the light gray region to $\mathcal{Y}_2$ (centers in $-\Sigma_2$).
In this configuration, the two regions do not intersect. In general, however, $\mathcal{Y}_1$ and $\mathcal{Y}_2$ do not have to be disjoint, see \autoref{lem:intY}.}
	\label{fig:symeta}
\end{figure}

Let $\phi\coloneqq\mathcal{F}_df$ denote the $d$-dimensional Fourier transform of the scattering potential $f$. According to \cite{ElbNauSch26_preprint}, the function $\phi$ is then related via the Fourier diffraction relation
\begin{equation}\label{eq:sys_eq0}
	\hat m(\eta, \sigma) =
	\begin{cases}
		a(\sigma) \phi( \eta - \sigma ) & \text{if } \sigma \in \Sigma_1, \vspace{2mm} \\
		a(\sigma) \phi( \eta - \sigma ) + a(H_\nu\sigma) \phi( \eta - H_\nu\sigma ) & \text{if } \sigma \in \Sigma_2
	\end{cases}
\end{equation}
for all $\eta\in S_{e_d}$, with $e_d$ denoting the $d$th standard basis vector in $\R^d$, to the reduced measurements $\hat m\colon S_{e_d}\times S_\omega\to\C$ which are with some explicitly known prefactor $C\colon S_{e_d}\times S_\omega\to\C$ given by
\[
\hat m(\eta,\sigma)\coloneqq C(\eta,\sigma)\int_{\nu^\perp}\int_{\R^{d-1}\times\{L\}}m(x,y) e^{-i\inner{\eta}{x}} e^{i\inner{\sigma}{y}}dS(x)dS(y).
\]
Since the Herglotz density $a$ is assumed to be given and the reduced measurements $\hat m$ can be directly calculated from the measurements $m$, \autoref{eq:sys_eq0} is a linear equation system for the values of the Fourier transform $\phi$ on the set $\mathcal Y_1\cup\mathcal Y_2$ where
\begin{equation}\label{eq:Y1_Y2}
	\mathcal{Y}_1\coloneqq\{\eta-\sigma \mid \eta\in S_{e_d},\ \sigma\in\Sigma_1\}\quad\text{ and } \quad \mathcal{Y}_2\coloneqq\left\lbrace \eta-\sigma \mid \eta\in S_{e_d}, \ \sigma\in\Sigma_2\right\rbrace
\end{equation}
correspond to the points appearing in the first and second case of \autoref{eq:sys_eq0}, respectively.

Geometrically, the sets $\mathcal{Y}_1$ and $\mathcal Y_2$ can be seen as unions of hemispheres centered at $-\Sigma_1\coloneqq\{-\sigma\mid \sigma\in\Sigma_1\}$ and $-\Sigma_2\coloneqq\{-\sigma\mid \sigma\in\Sigma_2\}$, see \autoref{subfig:union}.

\subsection{Reconstructibility of the scattering potential}
Unless we are in the case where $\Sigma_2=\emptyset$, that is, if $\nu$ is parallel to $\omega$ so that we are shifting the incident beam perpendicular to the incident direction~$\omega$, \autoref{eq:sys_eq0} is only an implicit equation for the values of the Fourier transform $\mathcal F_df$ of the scattering potential $f$. Even though we enforced with \autoref{eq:supp_a} that none of the coefficients of the linear system vanishes, it is a priori not clear if this system allows us to recover $\mathcal F_df$.

In principle, we could first recover the Fourier coefficients obtained from the first, direct relation in \autoref{eq:sys_eq0} and then use analytic continuation to determine the remaining ones, since $f$ is compactly supported and hence $\mathcal F_df\in C^\infty(\R^d;\C)$ is an analytic function. However, such arguments have limited practical value, as real measurement data are discrete.

We therefore do not want to rely on analyticity and will not impose any regularity on the solutions $\phi\colon\R^d\to\C$ of \autoref{eq:sys_eq0} for now. The question is then for which points $y\in\R^d$ the values $\phi(y)$ are uniquely determined by \autoref{eq:sys_eq0}.

To get rid of one of the coefficients $a(\sigma)$ and $a(H_\nu\sigma)$ in \autoref{eq:sys_eq0} by normalizing it to one, we introduce for some arbitrarily chosen $\chi\in(0,1)$ the function $b\colon\Sigma_{2,\chi}\to\C\setminus\{0\}$ on the open neighborhood
\begin{equation}\label{eq:SigmaNeigh}
\Sigma_{2,\chi} \coloneqq \left\{ \tilde\chi \sigma \,\middle|\, 
\sigma \in \Sigma_2,\ \tilde\chi \in (1-\chi,\ 1+\chi)\right\} \subseteq \R^d
\end{equation}
	of $\Sigma_2$ by
	\begin{equation}\label{eq:defb}
		b(\tilde\chi\sigma) \coloneqq \frac{a(\sigma)}{a(H_\nu\sigma)}\text{ for all }\sigma\in\Sigma_2,\,\tilde\chi\in(1-\chi,1+\chi).
	\end{equation}
	(Hereby, we only extended $b$ to the open set $\Sigma_{2,\chi}$ to be able to differentiate it like a function on $\R^d$ instead of having to use covariant derivatives on the sphere.)

Let $\phi$ and $\tilde \phi $ be two solutions of \autoref{eq:sys_eq0}. Since $a(\sigma)\neq 0$ for every $\sigma\in\Sigma_1$ and $a(H_\nu\sigma)\neq0$ for every $\sigma\in\Sigma_2$, we can divide the first equation in \autoref{eq:sys_eq0} by $a(\sigma)$ and the second equation by $a(H_\nu\sigma)$. Then, the difference $g\coloneqq\phi-\tilde\phi$ satisfies the homogeneous system
\begin{equation}\label{eq:sys_eq} 
	0= \begin{cases}  g( \eta - \sigma) & \text{if } \sigma \in \Sigma_1, \vspace{2mm} \\ b(\sigma) g( \eta - \sigma) + g( \eta - H_\nu\sigma) & \text{if } \sigma \in \Sigma_2 \end{cases} 
\end{equation}
for all $\eta\in S_{e_d}$. Thus, we have for an arbitrary point $y\in\R^d$ that $\phi (y)= \tilde\phi(y)$ for all solutions $\phi$ and $\tilde\phi$ of \autoref{eq:sys_eq0} if and only if every solution of the homogeneous system satisfies $g(y)=0$.

The aim of this work is to characterize all such points $y\in\R^d$.
\begin{problem}\label{pr:main}
Let $d\in\N\setminus\{1\}$ be the dimension of the problem, $k_0>0$ be an arbitrary wave number, $\nu,\omega\in\Sp_{k_0}^{d-1}$ be two directions describing the scan normal and the beam direction, the sets $\Sigma_1,\Sigma_2\subseteq S_\omega$ be given by \autoref{def:Sigma}, the corresponding regions $\mathcal Y_1$ and $\mathcal Y_2$ be defined by \autoref{eq:Y1_Y2}, $\chi\in(0,1)$ be arbitrarily chosen, and $b\colon\Sigma_{2,\chi}\to\C\setminus\{0\}$ be a function on the open neighborhood $\Sigma_{2,\chi}$ from \autoref{eq:SigmaNeigh} with $b(\sigma)=b(\tilde\chi\sigma)$ for all $\sigma\in\Sigma_2$ and $\tilde\chi\in(1-\chi,1+\chi)$. 

What is the largest set $Y\subseteq\mathcal Y_1\cup\mathcal Y_2$ for which every solution $g\colon\R^d\to\C$ of \autoref{eq:sys_eq} fulfills $g(y)=0$ for all $y\in Y$?
\end{problem}

Clearly, the first equation in \autoref{eq:sys_eq} immediately ensures that $g=0$ on $\mathcal Y_1$.

\begin{lemma}\label{lem:uqS1}
	Suppose $g\colon \R^d\to \C$ solves \autoref{eq:sys_eq} in the setting of \autoref{pr:main}.
	Then, we have $g= 0$  on the entire set  $\mathcal Y_1$.
\end{lemma}

While the uniqueness on $\mathcal{Y}_1$ is trivial, the analysis on the set $\mathcal Y_2$ is more challenging. 

For clarity, we first summarize the main results regarding the \autoref{pr:main}. The detailed proofs and technical lemmas are developed in the subsequent sections. Our results in higher dimensions ($d>2$) rely on the continuity of the solution $g$ of the homogeneous problem, while we do not impose any regularity in the two-dimensional case.
We denote by $C^p(U; \C)$ for an open subset $U\subseteq\R^d$ the space of $p$ times continuously differentiable functions $g\colon U\to\C$. And we use for an arbitrary $\sigma\in\R^d\setminus\{0\}$ the orthogonal projection
\begin{equation}\label{eq:projection}
\pi_\sigma\colon\R^d\to\R^d,\qquad\pi_\sigma x\coloneqq x-\inner x\sigma\frac{\sigma}{\norm\sigma^2},
\end{equation}
onto $\sigma^\perp$.

\begin{description}
	\item[Higher dimensions ($d>2$):]  In dimensions $d>2$, every solution $g\in C(\R^d;\C)$ of \autoref{eq:sys_eq} vanishes on the entire set $\mathcal Y_1\cup\mathcal{Y}_2$ provided that $b\in C^2(\Sigma_{2,\chi};\C)$ and
	\[
	\overline{\left\lbrace \sigma\in \Sigma_2 \mid \nabla b(\sigma)\notin\C\pi_\sigma\nu \right\rbrace } = \Sigma_2.
	\]
	This result is formally stated and proved for $d>3$ in \autoref{thm:uq2} and for $d=3$ in \autoref{thm:uq3}.
	\item[Two dimensions ($d=2$):] In two dimensions, we show that 
	every solution $g\colon\R^2\to\C$ of \autoref{eq:sys_eq} (not necessarily continuous anymore) vanishes almost everywhere on the subset $\mathcal Y_1\cup\tilde{\mathcal Y}$,
	\[
	\tilde{\mathcal{Y}} \coloneqq \left\lbrace \eta - H_\nu\sigma \mid \eta \in (-\Sigma_1) \cap S_{e_2},\ \sigma \in  \Sigma_2 \cap S_{-e_2},\  H_\nu\sigma\notin  S_{-e_2} \right\rbrace \subseteq \mathcal{Y}_2\setminus\mathcal{Y}_1;
	\]
	and that there exists for almost every point $y\in\mathcal Y_2\setminus\tilde{\mathcal Y}$ a solution $g\colon\R^2\to\C$ of \autoref{eq:sys_eq} with $g(y)\ne0$.
	This result is formalized in \autoref{thm:uq1}. 
\end{description}

	\section{Structure of the System through Coupling Sets}\label{sec:couplingset}
We start the analysis of \autoref{pr:main} by examining the second equation of \autoref{eq:sys_eq}, namely
\begin{equation}\label{eq:homY2}
0= b( \sigma )g( \eta - \sigma ) + g( \eta - H_\nu\sigma ), \qquad \ \eta\in S_{e_d}, \ \sigma\in\Sigma_2.
\end{equation}
A key observation is that the representation of a point $y\in\mathcal{Y}_2$ via $y = \eta - \sigma$ with $\sigma\in\Sigma_2$ and $\eta\in S_{e_d}$ is, in general, not unique. Different pairs $(\eta, \sigma)$ can produce the same point $y$, but lead to distinct coupled points $\eta- H_\nu\sigma$.
\begin{definition}\label{def:Fy}
	For $y \in \R^d$  we introduce the set $F_y \subseteq \R^d$ of all points $z\in \R^d$ for which there exist values $\eta \in S_{e_d}$ and $\sigma\in\Sigma_2$ with
	\[
	y = \eta - \sigma \quad \text{and} \quad z = \eta - H_\nu\sigma.
	\]
	We call $F_y$ the \emph{coupling set} of $y$.
\end{definition}

In other words, $F_y$ collects all points $z$ whose images $g(z)$ show up together with $g(y)$ in at least one equation of the system. By this, we can group all equations containing the term $g(y)$.

Note that we have $F_y\ne\emptyset$ if and only if $y\in\mathcal Y_2$, as $\mathcal Y_2$ is by definition the set of all points $y\in\R^d$ of the form $y=\eta-\sigma$ with $\eta\in S_{e_d}$ and $\sigma\in\Sigma_2$ which then leads to the point $\eta-H_\nu\sigma\in F_y$.
The following lemma provides an explicit characterization of the coupling set $F_y$. 
\begin{lemma}\label{lem:Fy}
	In the setting of \autoref{pr:main}, we have for the coupling set $F_y$ of an arbitrary point $y\in\R^d$ the identity
	\begin{equation}\label{eq:Fy}
		F_{y} = \left\lbrace y+2\inner{\sigma}{\nu}\nu\mid \sigma\in\Sigma_2\cap (S_{e_d}-y)\right\rbrace, 
	\end{equation}
	where $S_{e_d}-y=\{\sigma-y\in \R^d\mid \sigma\in S_{e_d}\}\subseteq \R^d$ denotes the upper hemisphere centered at $-y$.
\end{lemma}
\begin{proof}
	Let $y\in\R^d$. The set $M_y$ of all points $\sigma\in\Sigma_2$ for which there exists a value $\eta\in S_{e_d}$ satisfying $y=\eta-\sigma$ can be written as the intersection $$M_y= \Sigma_2\cap(S_{e_d}-y).$$ The points $z$ in $F_y$ are then characterized by 
	\[
	z = (y+\sigma)-H_\nu\sigma = y+2 \inner{\sigma}{\nu}\nu\quad \text{with } \sigma\in M_y,
	\]
	that is, \autoref{eq:Fy}.
\end{proof}

Fix $y\in\R^d$ and let $z\in F_y$. By definition of $F_y$, there exists at least one $\sigma\in\Sigma_2\cap(S_{e_d}-y)$ with \[z-y=2\inner\sigma\nu\nu\] and the corresponding value $\eta=\sigma+y$. For such a choice of $\sigma$ we have 
\[
0= b(\sigma)g(y) +g(z).
\]
In general, however, different $\sigma$ may correspond to the same pair of points $(y,z)$, leading to different coefficients $b(\sigma)$, and thus to potentially distinct relations between $g(y)$ and $g(z)$.

\begin{definition}\label{def:C}
	For every $y\in\R^d$ and every $\lambda\in \R$, we define the set
	\begin{equation}\label{eq:C}
		C_{y,\lambda} \coloneqq \left\{ \sigma \in \Sigma_2 \cap (S_{e_d}-y) \mid \inner{\sigma}{\nu} = \lambda \right\} 
	\end{equation}
	of all points in $\Sigma_2 \cap (S_{e_d}-y)$ lying in the plane perpendicular to $\nu$ with distance $\lambda$ to the origin.
\end{definition}

With these definitions, \autoref{eq:homY2} can equivalently be written as
\begin{equation}\label{eq:homFy}
	 0 = b(\sigma)g(y)+g(z)\quad\text{for all}\quad y\in\R^d,\,z\in F_y,\,\sigma\in C_{y,\frac12\inner{z-y}\nu}.
\end{equation}

\begin{remark}\label{rem:F_C_dimension}
In view of \autoref{eq:homFy}, we see the fundamental difference between the different dimensions:
\begin{itemize}
\item
For $d>2$, the intersection of $\Sigma_2$ and $S_{e_d}-y$ in the characterization of $F_y$ in \autoref{eq:Fy} is, in the case where the intersection is transversal and non-trivial, a relatively open subset of a $(d-2)$-dimensional sphere. Consequently, $F_y$ contains infinitely many points and the value $g(y)$ is involved in a continuum of equations in \autoref{eq:homFy}.

If $d>3$ and the relatively open subsets $\Sigma_2$ and $S_{e_d}-y$ of two hyperspheres intersect transversally and non-trivially with the hyperplane $\nu^\perp+\lambda\nu$, $\lambda\coloneqq\frac12\inner{z-y}\nu$ as in \autoref{eq:C}, then their intersection $C_{y,\lambda}$ is a $(d-3)$-dimensional manifold. In particular, it contains infinitely many points, too.
This means that we can expect to have for most pairs $(y,z)$ with $z\in F_y$ infinitely many equations in \autoref{eq:homFy} relating the terms $g(y)$ and $g(z)$.

\item
In three dimensions, the set $C_{y,\lambda}$ is the intersection of the subset $\Sigma_2\cap(S_{e_d}-y)$ of a circle with a plane, and thus consists, in the non-degenerate case, of at most two points. For a fixed pair $(y,z)$ of points, we will therefore, in general, not have multiple equations relating the two terms $g(y)$ and $g(z)$.

But we will still have for most points $y\in\mathcal Y_2$ that $g(y)$ appears in infinitely many equations in \autoref{eq:homFy}.

\item
In two dimensions, finally, the set $F_y$ contains according to \autoref{lem:Fy} (since $\Sigma_2$ and $S_{e_2}-y$ are two circular arcs) at most two points unless $y=0$. Hence, the term $g(y)$ can appear for $y\ne0$ in at most two equations from \autoref{eq:homFy}.

This difference of the coupling set $F_y$ between the case $d=2$ and $d>2$ is illustrated in \autoref{fig:Fy}.
\end{itemize}
\end{remark}

We will therefore study the unique solvability of \autoref{eq:homFy} in the cases $d=2$, $d=3$, and $d>3$ separately.

\definecolor{myblue}{rgb}{0.4,0.7,1}
\begin{figure}
	\centering
	\begin{subfigure}{0.45\textwidth}
		\includegraphics{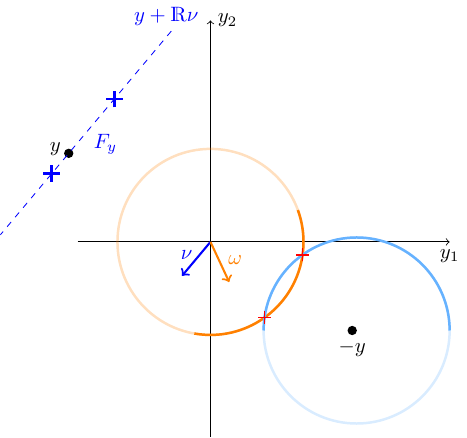}
		\caption{$F_y$ contains discrete points when $d=2$.}
	\end{subfigure}\hfill
	\begin{subfigure}{0.45\textwidth}
	\includegraphics{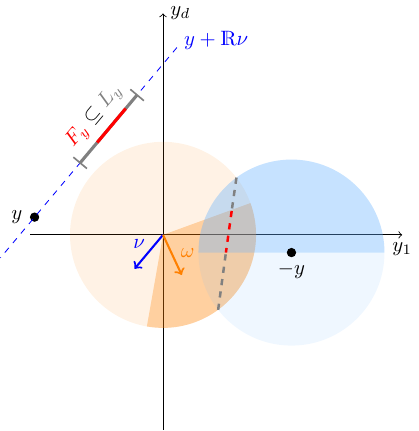}
		\caption{$F_y$ forms an open subset of the line segment $L_y$ when $d>2$.}
	\end{subfigure}
	\caption{Illustration of the coupling set $F_y$ defined in \autoref{eq:Fy}. The full spheres $\mathbb{S}^{d-1}_{k_0}$ and $\mathbb{S}^{d-1}_{k_0}-y$ are shown in light orange and light blue, respectively, while the subsets $\Sigma_2$ and $S_{e_d}-y$ are highlighted in saturated colors. (a) For $d=2$, the intersection $\Sigma_2\cap(S_{e_d}-y)$ consists of at most two points, whose projection onto the line $y+\mathbb{R}\nu$ yields the coupling set $F_y$. (b) For $d>2$, the dashed gray curve represents $\mathbb{S}^{d-1}_{k_0}\cap(\mathbb{S}^{d-1}_{k_0}-y)$, while the dashed red curve highlights the restricted intersection $\Sigma_2\cap(S_{e_d}-y)$. Their projections onto $y+\mathbb{R}\nu$ give the segment $L_y$ and the coupling set $F_y$, respectively.}
	\label{fig:Fy}
\end{figure}

\subsection{Graph representation of the system in two dimensions}
We first consider the two dimensional setting in \autoref{pr:main}. In this case, the term $g(y)$ of each point $y\in\mathcal Y_2\setminus\{0\}$ appears in the system in \autoref{eq:homFy} in at most two equations. Unless we get two equations coupling $g(y)$ twice with the same term $g(z)$, these equations alone are not sufficient to determine all occurring values uniquely. We therefore seek to combine these equations to a system with the same amount of unknowns and equations.

To visualize these relations in a clear and structured way, we will make use of graph representations. More precisely, we view the set $\mathcal Y_2$ as vertices and we draw an edge between two points whenever there exists an equation that directly couples their images. In other words, two vertices $y,z\in\mathcal{Y}_2$ are connected by an edge exactly when $z\in F_y$.

\begin{definition}\label{def:graph}
	We call $(\mathcal Y_2,E)$ the graph associated with the equation system in \autoref{eq:homFy}, where the vertices are the points in $\mathcal Y_2$ and the edges are given by
	\[ E\coloneqq\{\{y,z\}\in\mathcal Y_2\times\mathcal Y_2\mid z\in F_y\}. \]
	Moreover, a point $z\in \mathcal{Y}_2$ is called a neighbor of $y$ if $\{y, z\}\in E$.
\end{definition}

For a given point $y\in\mathcal Y_2$, we can explicitly determine its neighboring points.

\begin{lemma}\label{lem:By2}
	We consider the setting of \autoref{pr:main} for $d=2$. For all values $\eta\in S_{e_2}$ and $\sigma\in\Sigma_2$ with $\eta\neq\sigma$, the coupling set $F_y$ of the point $y\coloneqq\eta-\sigma$ is given by
	\begin{align}\label{eq:By2}
		F_{y} = \begin{cases}
			\left\lbrace \eta-H_\nu\sigma, H_\nu\eta-\sigma\right\rbrace  \quad &\text{if }\eta\in (-\Sigma_2)\cap S_{e_2} , \ \sigma\in \Sigma_2\cap S_{-e_2}, \\
			\left\lbrace\eta-H_\nu\sigma\right\rbrace &\text{otherwise}. 
		\end{cases}
	\end{align}
\end{lemma}

\begin{proof}
	Let $y=\eta-\sigma$ with $\eta\in S_{e_2}$ and $\sigma\in\Sigma_2$.  The circles $\Sp^1_{k_0}$ and $\Sp^1_{k_0}-y$ intersect in at most two points, namely $\sigma=\eta-y$ and $-\eta = -\sigma-y$. (If $\eta = -\sigma$, the intersection reduces to $-\eta$.) Hence, we always have
	\[
	\Sigma_2\cap(S_{e_2}-y)\subseteq  \Sp_{k_0}^{1}\cap(\Sp_{k_0}^1-1) = \{\sigma,-\eta\}.
	\]
	The point $\sigma$ is always contained in $\Sigma_2\cap(S_{e_2}-y)$, while the second point $-\eta$ belongs to $\Sigma_2\cap(S_{e_2}-y)$ if and only if
	\[
	-\eta\in \Sigma_2  \quad \text{and}\quad -\sigma = -\eta+y\in S_{e_2}.
	\]
	We thus  obtain
	\[
	\Sigma_2\cap(S_{e_2}-y) = \begin{cases}
		\{\sigma,-\eta\}\qquad &\text{if } \eta\in(-\Sigma_2)\cap S_{e_2},\ \sigma\in\Sigma_2 \cap S_{-e_2},\\
		\{\sigma\}&\text{otherwise}.
	\end{cases}
	\]
	Using the representation of $F_y$ from \autoref{eq:Fy} completes the proof.
\end{proof}

Let us mention some special cases of sets $F_y$ for $y=\eta-\sigma$ with $\eta\in S_{e_2}$ and $\sigma\in\Sigma_2$:
\begin{itemize}
\item
If $\sigma\in\nu^\perp$ or if $\eta\in\nu^\perp\cap(-\Sigma_2)$ and $\sigma\in S_{-e_2}$, then $y\in F_y$. This means that one equation in \autoref{eq:homFy} is of the form $(1+b(\sigma))g(y)=0$.
\item
If we have $\eta\in\{-\sigma,H_\nu\sigma\}$, then $\eta-H_\nu\sigma=H_\nu\eta-\sigma$. Thus, $F_y$ consists at most of one point. Otherwise, the points $\eta-H_\nu\sigma$ and $H_\nu\eta-\sigma$ are distinct.
\item
If $\eta=\sigma$, that is, $y=0$, the coupling set is given by
		\[
		F_0 = \{2\inner{\sigma}{\nu}\nu \mid\sigma\in\Sigma_{2}\cap S_{e_2}\},
		\] 
		which forms an entire interval on the line spanned by $\nu$.
\end{itemize}
To avoid additional technicalities associated with these degenerate configurations, we will exclude all such points from the analysis of reconstructibility on $\mathcal{Y}_2$ in two dimension. They will nevertheless only influence the reconstructibility on a set of measure zero.

\begin{definition}\label{def:nondegY2}
Let $S\subseteq\R^2$ be the set of all points $\eta-\sigma$ with $\eta\in S_{e_2}$ and $\sigma\in\Sigma_2$ for which $\eta\in\nu^\perp$ or $\sigma\in\nu^\perp$. We call $\mathcal Y_2'\coloneqq\mathcal Y_2\cap\mathcal B^d_{2k_0}\setminus(S\cup\nu^\perp\cup\R\nu)$ the set of non-degenerate points in $\mathcal Y_2$.
\end{definition}

For such points $y\in\mathcal Y_2'$, we can guarantee that the points appearing in \autoref{lem:By2} are all distinct.
\begin{lemma}\label{lem:distPts}
We consider the setting of \autoref{pr:main} for $d=2$. Let $y\in\mathcal Y_2'$ be written in the form $y=\eta-\sigma$ with $\eta\in S_{e_2}$ and $\sigma\in\Sigma_2$.

We then have that all the points $\eta-\sigma$, $H_\nu\eta-\sigma$, $\eta-H_\nu\sigma$, and $H_\nu\eta-H_\nu\sigma$ are distinct and not zero.
\end{lemma}
\begin{proof}
For one of the points to be zero, we would need that $\eta\in\{\sigma,H_\nu\sigma\}$, which would imply that $y=\eta-\sigma=0$ or $y=\eta-H_\nu\sigma+2\inner\sigma\nu\nu=2\inner\sigma\nu\nu\in\R\nu$.

For two points which differ by an application of $H_\nu$ to one term to be equal, we would need that $\eta=H_\nu\eta$ or $\sigma=H_\nu\sigma$, which would require $\eta\in\nu^\perp$ or $\sigma\in\nu^\perp$.

For $H_\nu\eta-\sigma=\eta-H_\nu\sigma$, we would need that $\eta\in\{-\sigma,H_\nu\sigma\}$, which would imply $y=2\eta\in\Sp^{d-1}_{2k_0}$ or, as before, $y\in\R\nu$.

Finally, for $\eta-\sigma=H_\nu(\eta-\sigma)$, we would require $y=\eta-\sigma\in\nu^\perp$.
\end{proof}

With the characterization of the neighbors of a vertex $y\in\mathcal{Y}_2$ presented in \autoref{lem:By2}, we can now analyze how the connected component of an arbitrary non-degenerate vertex looks like.

\begin{lemma}\label{lem:Zhk}
We consider the setting of \autoref{pr:main} for $d=2$. Each connected component $(C,E_C)$ of the graph $(\mathcal Y_2,E)$ associated to the equation system defined in \autoref{eq:homFy} containing a vertex from $\mathcal Y_2'$ has exactly one of the following forms:
	\begin{enumerate}
		\item
		\parbox[t]{0.3\textwidth}{\mbox{}\\[-2ex]\includegraphics{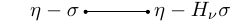}}
		for some $\eta\in S_{e_2}$ and $\sigma\in\Sigma_2$ with $\eta\notin(-\Sigma_2)$ or $\sigma,H_\nu\sigma\notin S_{-e_2}$;
		\item
		\parbox[t]{0.3\textwidth}{\mbox{}\\[-2ex]\includegraphics{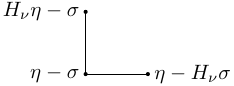}}
		\parbox[t]{0.5\textwidth}{for some $\eta\in(-\Sigma_2)\cap S_{e_2}$ with $H_\nu\eta\notin S_{e_2}$\\and  $\sigma\in\Sigma_2\cap S_{-e_2}$ with $H_\nu\sigma\notin S_{-e_2}$;}
		\item
		\parbox[t]{0.3\textwidth}{\mbox{}\\[-2ex]\includegraphics{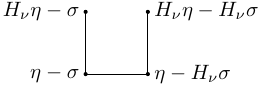}}
		\parbox[t]{0.5\textwidth}{for some $\eta\in(-\Sigma_2)\cap S_{e_2}$ with $H_\nu\eta\notin S_{e_2}$\\and $\sigma\in\Sigma_2\cap S_{-e_2}$ with $H_\nu\sigma\in S_{-e_2}$;}
		\item
		\parbox[t]{0.3\textwidth}{\mbox{}\\[-2ex]\includegraphics{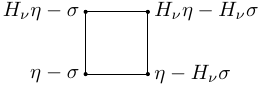}}
		\parbox[t]{0.5\textwidth}{for some $\eta\in(-\Sigma_2)\cap S_{e_2}$ with $H_\nu\eta\in S_{e_2}$\\ and $\sigma\in\Sigma_2\cap S_{-e_2}$ with $H_\nu\sigma\in S_{-e_2}$.}
		
	\end{enumerate}
\end{lemma}

\begin{proof}
	Let $(C,E_C)$ be a connected component of the graph $(\mathcal Y_2,E)$ containing a non-degenerate point $y\in C\cap\mathcal Y_2'$, which we write in the form $y=\eta-\sigma$ with $\eta\in S_{e_2}$ and $\sigma\in\Sigma_2$.
	
	According to \autoref{lem:By2}, the set $F_y$ of neighboring points of $y$ fulfills $F_y\subseteq\{z_1,z_2\}$ where $z_1\coloneqq\eta-H_\nu\sigma$ and $z_2\coloneqq H_\nu\eta-\sigma$ are non-zero according to \autoref{lem:distPts}. We analogously have that $F_{z_1}\subseteq\{y,z_3\}$. If $z_2\in\mathcal Y_2$, then also $F_{z_2}\subseteq\{y,z_3\}$, where $z_3\coloneqq H_\nu\eta-H_\nu\sigma$ is again non-zero. Thus, the second neighborhood of $y$ is a subset of $N_y\coloneqq\{y,z_1,z_2,z_3\}$. Moreover, since we find in the same way that $F_{z_3}\subseteq\{z_1,z_2\}$ if $z_3\in\mathcal Y_2$, we see that $C\subseteq N_y$ and $E_C\subseteq\{\{y,z_1\},\{y,z_2\},\{z_1,z_3\},\{z_2,z_3\}\}$.
	
	Since there are no isolated points in the graph $(\mathcal Y_2,E)$ and we know from \autoref{lem:distPts} that all the four points $y$, $z_1$, $z_2$, and $z_3$ are distinct, this leaves exactly the four possible connected components listed in the proposition.
	To characterize which vertices correspond to which of these cases, we pick the reference point $\eta-\sigma$ with $\eta\in S_{e_2}$ and $\sigma\in\Sigma_2$ as shown in the graphs and  check for each edge under which condition it exists. 
	\begin{itemize}
		\item
		By applying \autoref{lem:By2} to $y=\eta-\sigma$, we see that the edge $\{y,z_1\}$ always exists and
		\item
		that the edge $\{y,z_2\}$ exists if and only if $\eta\in-\Sigma_2$ and $\sigma\in S_{-e_2}$.
		\item
		Applying \autoref{lem:By2} to $z_1=\eta-H_\nu\sigma$, we find that the edge $\{z_1,z_3\}$ exists if and only if $\eta\in-\Sigma_2$ and $H_\nu\sigma\in S_{-e_2}$.
		\item Finally, assuming that the edge $\{y,z_2\}$ exists, we get by applying \autoref{lem:By2} to $z_2=-\sigma+H_\nu\eta$ that the edge $\{z_2,z_3\}$ exists if and only if $H_\nu\eta\in S_{e_2}$. 
		
	\end{itemize}
Putting this together, we get the asserted conditions.
\end{proof}

Thus, the graph associated with \autoref{eq:homFy} decomposes into connected components, and the equation system can be studied separately on each of them.
If, on a given component, the number of vertices exceeds the number of edges (as in the first three cases in \autoref{lem:Zhk}), then the induced linear system for the values of $g$ on this component is underdetermined. Consequently, uniqueness cannot be expected unless the value at one of the vertices is already known to be zero from the first part of \autoref{eq:sys_eq}. 
If the numbers coincide as in the last case in \autoref{lem:Zhk}, the component gives rise to a linear system with equally many equations as unknowns which could potentially be uniquely solvable.

We discuss both possibilities further in \autoref{sec:two}.

\subsection{Properties of the coupling sets in higher dimensions}
For $d > 2$, the coupling sets $F_y$ can contain a continuum of points. First, we show that each non-empty $F_y$ is forms a relatively open subset of a closed line segment $L_y$ on a line parallel to $\nu$.

\begin{lemma}\label{lem:Ly}
We consider the setting of \autoref{pr:main} for $d>2$. For every point $y\in\R^d$, the coupling set $F_y$ is a relatively open subset of the closed line segment
	\begin{equation}\label{eq:Ly}
	L_y \coloneqq\left\lbrace y+2\inner{\sigma}{\nu}\nu\mid \sigma\in \Sp^{d-1}_{k_0}\cap(\Sp^{d-1}_{k_0}-y)\right\rbrace  = \left\lbrace y+2\lambda\nu\mid \lambda\in\Lambda_y\right\rbrace ,
	\end{equation}
	with the closed interval
	\begin{equation}\label{eq:Lambday}
	\Lambda_y\coloneqq \left\lbrace \inner{\sigma}{\nu}\mid \sigma\in \Sp^{d-1}_{k_0}\cap(\Sp^{d-1}_{k_0}-y)\right\rbrace.
	\end{equation}
\end{lemma}

\begin{proof}
	It is clear from the explicit representation in \autoref{eq:Fy} that $F_y$ is a subset of $L_y$.
	
	For every point $z\in F_y$, we then find a point $\sigma\in \Sigma_2\cap(S_{e_d}-y)$ with
	\[
	z = y+2\lambda_0\nu, \quad \lambda_0\coloneqq\langle\sigma,\nu\rangle.
	\]
	Since the set 
	\[ \Sigma_2\cap(S_{e_d}-y)= \mathbb{S}^{d-1}_{k_0} \cap (\mathbb{S}^{d-1}_{k_0}-y)\cap \Omega_\omega\cap \Omega_{H_\nu\omega}\cap (\Omega_{e_d}-y)\]
	is relatively open in the circle $C\coloneqq \Sp^{d-1}_{k_0}\cap(\Sp^{d-1}_{k_0}-y)$, we find a ball $\mathcal{B}^{d}_\rho(\sigma)$ with  $\mathcal{B}^{d}_\rho(\sigma)\cap C\subseteq\Sigma_2\cap (S_{e_d}-y)$. We thus have
	\begin{align*}
		L_y\cap\{y+2\lambda\nu \mid \lambda\in(\lambda_0-\rho,\lambda_0+\rho)\} &= \{y+2\lambda\nu \mid \lambda\in(\lambda_0-\rho,\lambda_0+\rho)\cap \Lambda_y\}\\ &= \{y+2\langle\tilde{\sigma},\nu\rangle\nu \mid \tilde{\sigma}\in \mathcal{B}^{d}_\rho(\sigma)\cap C\}\subseteq F_y,
	\end{align*}
	showing that $z$ is an inner point of $F_y\subseteq L_y$.
\end{proof}

Next, we investigate how the coupling sets of two base points behave when the points are sufficiently close to each other.

\begin{lemma}\label{lem:openBy}
	We consider the setting of \autoref{pr:main} for $d>2$. Let $y\in \mathcal B^d_{2k_0}$ be a point for which the coupling set $F_y$ is not empty.

	Then, there exists a point $\tilde{y}\in\mathcal B^d_{2k_0}\cap(y+\R\nu)$ such that $F_y\cap F_{\tilde{y}}$
 contains a non-empty open line segment on $y+\R\nu$.
\end{lemma}
\begin{proof}
Let us introduce the sets $L_y$ and $\Lambda_y$ as in \autoref{lem:Ly} and define the interior $\mathring{L}_y\coloneqq\{y+2\lambda\nu\mid\lambda\in\mathring\Lambda_y\}$ of $L_y$ in $y+\R\nu$, which is not empty because of the assumption $y\in \mathcal B^d_{2k_0}$.

Since $F_y$ is according to \autoref{lem:Ly} a relatively open subset of the line segment $L_y$, we can find a point $z\in F_y\cap\mathring{L}_y$. For this point, we then have a $\sigma\in\Sigma_2\cap(S_{e_d}-y)$ with
	\[z = y +2\lambda\nu, \qquad \lambda \coloneqq \inner{\sigma}{\nu}\in\mathring{\Lambda}_y. \] 
Furthermore, the openness of $F_y$ allows us to find a parameter $\varepsilon>0$ with $y+2\tilde{\lambda}\nu\in F_y$ for all $\tilde \lambda\in(\lambda-\varepsilon,\lambda+\varepsilon)$.
	
Since the spheres $\Sp^{d-1}_{k_0}$ and $\Sp^{d-1}_{k_0}-y$ intersect transversally for $y\in \mathcal{B}^{d}_{2k_0}$, the intersection depends continuously on $y$. That is, we can choose $\tilde{y}\in\mathcal{B}^{d}_{\varepsilon}(y)\cap(y+\R\nu)$ close to $y$ with $\mathring{L}_y\cap\mathring{L}_{\tilde{y}}\neq\emptyset$ such that
	\[
	\operatorname{dist}(\Sigma_2\cap(S_{e_d}-y), \Sigma_2\cap(S_{e_d}-\tilde y))<\frac{\varepsilon}{2}.
	\]
In particular, we can find a point $\tilde{\sigma}\in \Sigma_2\cap(S_{e_d}-\tilde y)$ with $\norm{\sigma-\tilde{\sigma}}<\frac{\varepsilon}{2}$. Defining then the point
	\[
	\tilde{z}\coloneqq\tilde{y}+2\inner{\tilde{\sigma}}{\nu}\nu = z+(\tilde{y}-y)+2\inner{\tilde{\sigma}-\sigma}{\nu}\nu\in F_{\tilde{y}},
	\]
and remarking that $\frac12(\tilde{y}-y)+\inner{\tilde{\sigma}-\sigma}{\nu}\nu = \hat\lambda\nu$ for some $\hat\lambda\in(-\varepsilon,\varepsilon)$, we find that $\tilde z = y+2\tilde{\lambda}\nu\in F_y\cap F_{\tilde{y}}$ with $\tilde{\lambda}\coloneqq\lambda+\hat\lambda\in (\lambda-\varepsilon,\lambda+\varepsilon)$.

Thus, $F_y\cap F_{\tilde y}$ is a non-empty, relatively open subset of the non-degenerate line segment $L_y\cap L_{\tilde y}$.
\end{proof}

\section{Reconstructible Fourier coefficients in dimensions larger than three}\label{sec:higher}
In this section, we study the \autoref{pr:main} in the case $d>3$, exploiting the additional structure provided by the sets $C_{y,\lambda}$ introduced in \autoref{def:C}. Recall that for fixed $y\in\R^d$ and $z\in F_y$ the \autoref{eq:homFy} yields the family
\[
0 = b(\sigma) g(y)+g(z)\qquad \text{for all } \sigma\in C_{y,\frac{1}{2}\inner{z-y}{\nu}}
\]
of equations. If at least two of these equations are linearly independent, then $g(y)=g(z)=0$.

\begin{proposition}\label{prop:uniquel3}
We consider the setting of \autoref{pr:main} for $d>3$. Let $y\in\mathcal Y_2$ and suppose that $g$ solves the homogeneous system from \autoref{eq:sys_eq}. 

	If there exists a value $\lambda\in \R$ such that the function $b$ is not constant on the set $C_{y,\lambda}$, introduced in \autoref{def:C}, then $g(y)=0$.
\end{proposition}
\begin{proof}
	Since $d>3$ and, by assumption, $b$ is not constant on $C_{y,\lambda}$, there exist two distinct vectors $\sigma,\hat\sigma\in C_{y,\lambda}$ with $b(\sigma)\neq b(\hat\sigma)$. Then, the pair $(g(y),g(z))$ with $z\coloneqq y+2\lambda\nu$  satisfies the system 
	\begin{align*}
		B\begin{pmatrix}
			g(y)\\g(z)
		\end{pmatrix} = 0, \quad \text{with } B \coloneqq 
		\begin{pmatrix}
				b(\sigma) & 1\\
				b(\hat{\sigma}) & 1
			\end{pmatrix}.
	\end{align*}
	Since the determinant of $B$ is	
	\[
	\det(B)= b(\sigma)-b(\hat\sigma)\ne0,
	\]
	 the only solution is $g(y)=g(z)=0$.
\end{proof}

\autoref{prop:uniquel3} shows that $g(y) = 0$ for every $y\in\mathcal{Y}_2$ for which there exists a level $\lambda\in \R$ such that $b$ is non-constant on $C_{y,\lambda}$. If we impose that the solution $g$ shall be continuous, we can extend this point-wise result also to points $y$ which do not fulfill the assumptions in \autoref{prop:uniquel3}, but for which we find in every neighborhood of $y$ some point which does.

\begin{lemma}\label{th:C}
We consider the setting of \autoref{pr:main} for $d>3$. Let $y\in\mathcal Y_2\cap\mathcal B_{2k_0}^d\setminus\R\nu$. There then exist an open neighborhood $V\subseteq\R^d$ of $y$ and an open interval $I\subseteq\R$ such that the set $C_{y',\lambda'}$, introduced in \autoref{def:C}, is a $(d-3)$-dimensional manifold for every $y'\in V$ and $\lambda'\in I$.

Its tangent space is given by
\begin{equation}\label{eq:TC}
\mathrm T_{\sigma'}C_{y',\lambda'} = \big(\operatorname{span}\{y',\sigma',\nu\}\big)^\perp \quad \text{for every }\sigma'\in C_{y',\lambda'}.
\end{equation}
\end{lemma}
\begin{proof}
For an arbitrary value $\lambda\in\R$, the set $C_{y,\lambda}$ is by definition the intersection of the relatively open subset $M_y\coloneqq\Sigma_2\cap(S_{e_d}-y)$ of $\mathbb S_{k_0}^{d-1}\cap(\mathbb S_{k_0}^{d-1}-y)$ and the hyperplane $N_\lambda\coloneqq\{\sigma\in\R^d\mid \inner\sigma\nu=\lambda\}$. Since $y\in\mathcal Y_2$, we have that $M_y\ne\emptyset$, which is thus a $(d-2)$-dimensional submanifold of the $(d-2)$-dimensional sphere
\[ \mathbb S_{k_0}^{d-1}\cap(\mathbb S_{k_0}^{d-1}-y)=\left\{\sigma\in\R^d\;\middle|\;\inner{\sigma-\frac y2}y=0,\,\norm{\sigma-\frac y2}=r_y\right\} \]
with radius $r_y\coloneqq\sqrt{k_0^2-\frac14\norm y^2}\in(0,k_0)$ around $\frac y2$ in a plane orthogonal to $y$. The orthogonal complement of its tangent space at a point $\sigma\in\mathbb S_{k_0}^{d-1}\cap(\mathbb S_{k_0}^{d-1}-y)$ is thus given by $\operatorname{span}\{y,\sigma-\frac y2\}=\operatorname{span}\{y,\sigma\}$.

We denote by $\nu_y\coloneqq\nu-\inner\nu y\frac y{\norm y^2}$ the orthogonal projection of $\nu$ onto the orthogonal complement of $y$. Since $y\notin \R\nu$, we have that $\nu_y \neq 0$. For every point
\[ \sigma\in\tilde M_y\coloneqq M_y\setminus\operatorname{span}\{y,\nu\}= M_y\setminus\left\{\frac y2-\frac{r_y}{\norm{\nu_y}}\nu_y,\frac y2+\frac{r_y}{\norm{\nu_y}}\nu_y\right\} \]
 the vectors $y$, $\sigma$, and $\nu$ are linearly independent, which implies that the tangent spaces of $\tilde M_y$ and $N_\lambda$ fulfill $\mathrm T_\sigma\tilde M_y+\mathrm T_\sigma N_\lambda=(\operatorname{span}\{y,\sigma\})^\perp+\nu^\perp=\R^d$ at every point $\sigma\in\tilde M_y\cap N_\lambda$.

Picking thus an arbitrary element $\sigma_y\in\tilde M_y$ and setting $\lambda_y\coloneqq\inner{\sigma_y}\nu$, we have
\[ \inner{\frac y2-\frac{r_y}{\norm{\nu_y}}\nu_y}\nu < \lambda_y < \inner{\frac y2+\frac{r_y}{\norm{\nu_y}}\nu_y}\nu \]
so that $C_{y,\lambda_y}=M_y\cap N_{\lambda_y}\subseteq\tilde M_y$ and therefore $\mathrm T_\sigma M_y+\mathrm T_\sigma N_{\lambda_y}=\R^d$ for every $\sigma\in C_{y,\lambda_y}$. This means that the manifolds $M_y$ and $N_{\lambda_y}$ intersect transversally and we therefore find open neighborhoods $V\subseteq\R^d$ of $y$ and $I\subseteq\R$ of $\lambda_y$ such that $M_{y'}$ and $N_{\lambda'}$ intersect transversally in a non-empty set $C_{y',\lambda'}$ for every $y'\in V$ and $\lambda'\in I$, which implies that $C_{y',\lambda'}$ is a $(d-3)$-dimensional manifold.

Moreover, its tangent space $\mathrm T_{\sigma'}C_{y',\lambda'}$ at a point $\sigma'\in C_{y',\lambda'}$ is the intersection of the tangent spaces $\mathrm T_{\sigma'}M_{y'}=(\operatorname{span}\{y',\sigma'\})^\perp$ and $\mathrm T_{\sigma'}N_{\lambda'}=\nu^\perp$, which yields \autoref{eq:TC}.
\end{proof}

It is not a big limitation that this statement is restricted to the subset $\mathcal Y_2\cap\mathcal B_{2k_0}^d\setminus\R\nu$ of $\mathcal Y_2$ as this subset is dense.

\begin{lemma}
We consider the setting of \autoref{pr:main} for $d\ge2$. The set $\mathcal Y_2\cap\mathcal B_{2k_0}^d\setminus\R\nu$ is open in $\R^d$ and dense in $\mathcal Y_2$.
\end{lemma}
\begin{proof}
Let $y\in\mathcal Y_2\cap(\mathbb S_{2k_0}^{d-1}\cup\{0\})$. We can then write $y=\eta-\sigma$ with two collinear vectors $\eta\in S_{e_d}$ and $\sigma\in\Sigma_2$. As $\Sigma_2$ is relatively open in $\mathbb S_{k_0}^{d-1}$, we can find a sequence $(\sigma_n)_{n\in\N}$ in $\Sigma_2\setminus\{\sigma\}$ converging to $\sigma$. The sequence $(y_n)_{n\in\N}$, $y_n\coloneqq\eta-\sigma_n$, in $\mathcal Y_2\cap\mathcal B_{2k_0}^d\setminus\{0\}$ then converges to $y$, which proves that $\mathcal Y_2\cap\mathcal B_{2k_0}^d\setminus\{0\}$ is dense in $\mathcal Y_2$.

To verify that $\mathcal Y_2\cap\mathcal B_{2k_0}^d\setminus\{0\}$ is open, we express $\mathcal Y_2$ as the set $\{y\in\R^d\mid\Sigma_2\cap(S_{e_d}-y)\ne\emptyset\}$. And since the manifolds $\Sigma_2\subseteq\mathbb S_{k_0}^{d-1}$ and $S_{e_d}-y\subseteq\mathbb S_{k_0}^{d-1}-y$ intersect transversally for every $y\in\mathcal B_{2k_0}^d\setminus\{0\}$, the intersection is also non-empty in a neighborhood of such a point $y$ which shows the openness of $\mathcal Y_2\cap\mathcal B_{2k_0}^d\setminus\{0\}$.

Therefore, also $\mathcal Y_2\cap\mathcal B_{2k_0}^d\setminus\R\nu$ is an open set which is dense in $\mathcal Y_2$.
\end{proof}

\begin{lemma}\label{lem:gradb}
We consider the setting of \autoref{pr:main} for $d>3$ and assume the additional regularity $b\in C^1(\Sigma_{2,\chi};\C)$ and the property
	\begin{equation}\label{eq:densityPropHigherDim}
	\overline{\{\sigma \in \Sigma_2 \mid\nabla b(\sigma)\notin\C\pi_\sigma\nu\}} = \Sigma_2,
	\end{equation}
	where $\pi_\sigma$ denotes the orthogonal projection onto $\sigma^\perp$ as defined in \autoref{eq:projection}.

	There then exists in every relatively open subset $U\subseteq\mathcal Y_2$ a point $y'\in U$ and a value $\lambda'\in\R$ such that $b$ is not constant on the set $C_{y',\lambda'}$, introduced in \autoref{def:C}.
\end{lemma}
\begin{proof}
Let $U\subseteq\mathcal Y_2$ be an arbitrary relatively open set. Since the open set $\mathcal Y_2\cap\mathcal B_{2k_0}^d\setminus\R\nu$ is dense in $\mathcal Y_2$, we can find a point $y\in U\cap\mathcal B_{2k_0}^d\setminus\R\nu$. By \autoref{th:C}, there exist an open neighborhood $V\subseteq U$ of $y$ and an open interval $I\subseteq\R$ such that $C_{y',\lambda'}$ is a $(d-3)$-dimensional manifold for every $y'\in V$ and $\lambda'\in I$.

Writing $y=\eta-\sigma$ for some $\eta\in S_{e_d}$ and $\sigma\in\Sigma_2$ with $\lambda\coloneqq\inner\sigma\nu\in I$, we can find, thanks to our assumption in \autoref{eq:densityPropHigherDim}, a point $\sigma'\in\Sigma_2$ with $\eta-\sigma'\in V$ and $\lambda'\coloneqq\inner{\sigma'}\nu\in I$ such that $\nabla b(\sigma')\notin\C\pi_\sigma\nu$. We can therefore pick $\tilde b\in\{\operatorname{Re}(b),\operatorname{Im}(b)\}$ either as the real or the imaginary part of $b$ such that $\nabla\tilde b(\sigma')$ and $\pi_\sigma\nu$ are linearly independent.

Since $\tilde b$ is by definition in \autoref{pr:main} constant along $\Sigma_{2,\chi}\cap\R\sigma'$, we moreover have that $\nabla\tilde b(\sigma')$ is orthogonal to $\sigma'$. Therefore, the three vectors $\nabla\tilde b(\sigma')$, $\sigma'$, and $\pi_\sigma\nu$ are linearly independent, which implies that also $\nabla\tilde b(\sigma')$, $\sigma'$, and $\nu$ are linearly independent.

Finally, we pick a point $\eta'\in S_{e_d}$ close to $\eta$ such that $y'\coloneqq\eta'-\sigma'\in V\setminus\operatorname{span}\{\nabla\tilde b(\sigma'),\sigma',\nu\}$ and obtain that
\[ \nabla\tilde b(\sigma')\notin\operatorname{span}\{y',\sigma',\nu\}=(\mathrm T_{\sigma'}C_{y',\lambda'})^\perp, \]
which means that $\tilde b$ and thus also $b$ cannot be constant on $C_{y',\lambda'}$.
\end{proof}

As direct consequence, we can now state one of our main results.
\begin{theorem}\label{thm:uq2}
	We consider the setting of \autoref{pr:main} for $d>3$ and assume the additional regularity $b\in C^1(\Sigma_{2,\chi};\C)$ and the property from \autoref{eq:densityPropHigherDim}.

	For every continuous solution $g\in C(\R^d;\C)$ of \autoref{eq:sys_eq}, we then have $g(y)=0$ for all $y\in\mathcal{Y}_1\cup\mathcal{Y}_2$. 
\end{theorem}
\begin{proof}
Let $g\in C(\R^d;\C)$ be an arbitrary solution of \autoref{eq:sys_eq}. The statement on $\mathcal{Y}_1$ follows from \autoref{lem:uqS1}.

	 Let further $y\in\mathcal{Y}_2$ be arbitrary, and let $U$ be a neighborhood of this point. Since the set of all $\sigma\in\Sigma_2$ for which $\nabla b(\sigma)\notin\C\pi_\sigma\nu$ is dense in $\Sigma_2$, there exist by \autoref{lem:gradb}, at least one point $y'\in U$ at which the assumption of \autoref{prop:uniquel3} is satisfied which gives $g(y')=0$.

	Consequently, we can construct a sequence $(y'_n)_{n \in \N} $ with $y'_n \to y$ and $g(y'_n) = 0$ for all $n\in\N$.
	By continuity of $g$, we conclude that
	\[
	g(y) = \lim_{n \to \infty} g(y_n') = 0.
	\]
\end{proof}

\section{Reconstructible Fourier coefficients in three dimensions}\label{subsec:3d}

If we reduce in \autoref{pr:main} the dimensions to $d=3$, it is no longer possible to choose for every point $y\in\mathcal Y_2$ a value $\lambda\in\R$ such that the set $C_{y,\lambda}$, defined in \autoref{def:C}, consists of a continuum of points. Hence, in contrast to the higher-dimensional case, we cannot obtain an infinite number of equations for the pair $(g(y),g(z))$ with $z\coloneqq y+2t\nu$ in \autoref{eq:homFy}.

Instead, as discussed in \autoref{rem:F_C_dimension}, the set $C_{y,\lambda}$ consists in the non-degenerate case of at most two points and for some $y\in\mathcal Y_2$, there may be no value $\lambda$ for which $C_{y,\lambda}$ has more than a single point.

To overcome this, we exploit \autoref{lem:openBy}. It states that the coupling set $F_{\tilde y}$ of a point $\tilde y\in\mathcal Y_2$  is a relatively open subset of a typically non-degenerate line segment. Hence, we can choose an additional point $\hat y\in\mathcal Y_2$ on the line $\tilde y+\R\nu$ such that the intersection of the coupling sets $F_{\tilde y}$ and $F_{\hat y}$ is a non-empty, relatively open subset of a non-degenerate line segment.  Selecting two distinct points $\tilde z,\hat z\in F_{\tilde y}\cap F_{\hat{y}}$ gives one equation for each of the pairs $(g(\tilde y),g(\tilde z))$, $(g(\tilde y),g(\hat z))$, $(g(\hat y),g(\tilde z))$, and $(g(\hat y),g(\hat z))$ in \autoref{eq:homFy}. If these equations are linearly independent, the only solution is $g(\tilde y)=g(\hat y)=g(\tilde z)=g(\hat z)=0$.

We now relate this back to the original point $y$: given $y\in\mathcal{Y}_2$ together with a point $z\in F_y$ in its coupling set, we show  that there typically exists in every neighborhood of $y$ a point $\tilde y=y+w\in\mathcal Y_2$ which we can complete with $\hat y=y+w+\delta\nu$, $\tilde z=z+w$, and $z+w+\varepsilon\nu$ to such a set of four points.

A visualization of this construction is provided in \autoref{fig:neighbor}.

\begin{lemma}\label{lem:pointsonline}
In the setting of \autoref{pr:main} for $d=3$, let $\eta\in S_{e_3}$ and $\sigma\in\Sigma_2$ be given such that the vectors $\eta$, $\sigma$ and $\nu$ are linearly independent and set $y\coloneqq\eta-\sigma$ and $z\coloneqq\eta-H_\nu\sigma$.

Then, for sufficiently small neighborhoods $U\subseteq\R^3$ of the origin and $J\subseteq\R$ of $0$, there exist smooth functions $\hat\eta\colon U\times J\times J\to S_{e_3}$ and $\hat{\sigma}\colon U\times J\times J\to\Sigma_2$ defined by $\hat\eta(0,0,0) = \eta$, $\hat\sigma(0,0,0)=\sigma$, 
	\begin{equation}\label{eq:hatsigma}
		y+w+\delta\nu = \hat\eta(w,\delta,\varepsilon) -\hat\sigma(w,\delta,\varepsilon), \quad \text{and } \quad
		z+w+\varepsilon\nu =\hat \eta(w,\delta,\varepsilon)- H_\nu\hat\sigma(w,\delta,\varepsilon).
	\end{equation}
\end{lemma}

\begin{proof}
	We seek a pair $(\hat{\eta}, \hat{\sigma})$ close to $(\eta, \sigma)$ that satisfies \autoref{eq:hatsigma} for small $(w, \delta, \varepsilon) \in U \times J \times J$. We subtract the first equation from the second and rewrite the system in the form $\Phi(\hat{\eta}, \hat{\sigma}; w, \delta, \varepsilon)=0$ with the function
	\[
	\Phi: \mathbb{S}^2_{k_0} \times \mathbb{S}^2_{k_0} \times U \times J \times J \to \R^3 \times \R,  \qquad \Phi(\hat{\eta}, \hat{\sigma}; w, \delta, \varepsilon) \coloneqq \begin{pmatrix}
		\hat{\eta}-\hat{\sigma} - (y+w+\delta\nu) \\
		2\inner{\hat{\sigma}}{\nu} - (2\inner{\sigma}{\nu} + \varepsilon-\delta)
	\end{pmatrix}.
	\]
	By construction, we know $\Phi(\eta,\sigma; 0,0,0) = 0$.
	
	The derivative of the unperturbed map $\tilde{\Phi}(\eta,\sigma)\coloneqq\Phi(\eta,\sigma; 0,0,0)$  acts on $(u,v)\in\mathrm T_\eta\Sp_{k_0}^2\times\mathrm T_\sigma \Sp_{k_0}^2$ as
	\[
	\mathrm D\tilde{\Phi}(\eta,\sigma)(u,v) =\begin{pmatrix}u-v\\2\inner{v}{\nu}\end{pmatrix}.
	\]
	Now, $\mathrm D\tilde{\Phi}(\eta,\sigma)(u,v) = 0$ can only happen if $u=v$ and  $\inner{v}{\nu}=0$. The condition $u=v$ additionally requires  $v\in T_\sigma\Sp_{k_0}^2\cap T_\eta \Sp_{k_0}^2$ so that
	\[
	v\perp \operatorname{span}\{\nu,\sigma,\eta\}.
	\]
	Because $\nu$, $\sigma$, and $\eta$ are assumed to be linearly independent, this is only possible for $v=u=0$. Hence the kernel of $\mathrm D\tilde{\Phi}(\eta,\sigma)$ is trivial.
	
	Thus $\mathrm D\tilde\Phi(\eta,\sigma)$ is an isomorphism between two $4$-dimensional spaces and, by the implicit function theorem, there therefore exists for small $(w,\delta,\varepsilon)$ a unique solution $(\hat{\eta}(w,\delta,\varepsilon),\hat\sigma(w,\delta,\varepsilon))$ of
\[ \Phi(\hat\eta(w,\delta,\varepsilon),\hat\sigma(w,\delta,\varepsilon);w,\delta,\varepsilon)=0\quad\text{with}\quad (\hat{\eta}(0,0,0),\hat\sigma(0,0,0))=(\eta,\sigma) \]
depending smoothly on $(w,\delta,\varepsilon)$.
\end{proof}

\begin{figure}[t]
	\centering
	\includegraphics{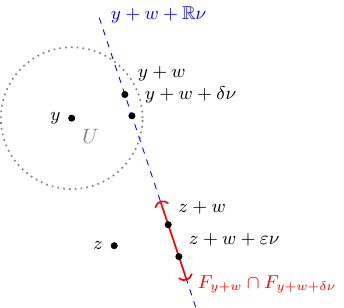}
	\caption{Illustration of the construction from \autoref{lem:pointsonline}. Starting with a point $y+w$ in a neighborhood $U$ of $y$, we perturb it slightly in the direction $\nu$ to obtain $y+w+\delta\nu$ for a sufficiently small $\delta>0$ such that the associated coupling sets $F_{y+w}$ and $F_{y+w+\delta\nu}$ intersect. From this intersection we select two points, $z+w$ and $z+w+\varepsilon\nu$. Together, the four points form a closed $4\times4$ system. }
	\label{fig:neighbor}
\end{figure}

\begin{corollary}\label{cor:closed}
In the setting of \autoref{pr:main} for $d=3$, let $\eta\in S_{e_3}$ and $\sigma\in\Sigma_2$ be given such that the vectors $\eta$, $\sigma$ and $\nu$ are linearly independent, $y\coloneqq\eta-\sigma$ and $z\coloneqq\eta-H_\nu\sigma$, and let $\hat\eta$ and $\hat\sigma$ be constructed as in \autoref{lem:pointsonline}.
Suppose further that $g\in C(\R^3;\C)$ is a continuous solution of \autoref{eq:sys_eq}. 

Then, for every $(w,\delta,\varepsilon) \in U \times J \times J$, the function $g$ satisfies
		\begin{align}\label{eq:mateq}
		B(w,\delta,\varepsilon) \begin{pmatrix}
			g(y+w)\\g(y+w+\delta\nu)\\g(z+w)\\g(z+w+\varepsilon\nu)
		\end{pmatrix} =0, 
	\end{align}
	 with
	\begin{equation}\label{eq:mat}B(w,\delta,\varepsilon) \coloneqq \begin{pmatrix}
			b(\hat\sigma(w,0,0)) & 0& 1& 0\\
			b(\hat\sigma(w,0,\varepsilon))& 0 &0 &1\\
			0& b(\hat\sigma(w,\delta,0))& 1 & 0\\
			0& b(\hat\sigma(w,\delta,\varepsilon))& 0 &1
		\end{pmatrix}.\end{equation}
\end{corollary}

\begin{proof}
	For each choice $(w,\delta,\varepsilon)\in U\times J\times J$, we obtain from the defining property
	of the functions $\hat\eta$ and~$\hat\sigma$, \autoref{eq:hatsigma}, the corresponding equation
	\[ 0 = b(\hat\sigma(w,\delta,\varepsilon))g(y+w+\delta\nu)+g(z+w+\varepsilon\nu) \]
	from the system in \autoref{eq:homY2}. Choosing the values $(w,0,0)$, $(w,0,\varepsilon)$, $(w,\delta,0)$, and $(w,\delta,\varepsilon)$, we thus have the four equations in \autoref{eq:mateq} for the four values of $g$ at the points $y+w$, $y+w+\delta\nu$, $z+w$, and $z+w+\varepsilon\nu$.
\end{proof}

\subsection{Analysis of the system determinant}\label{subsec:det} 

Taken together,  \autoref{lem:pointsonline} and \autoref{cor:closed} show that for points 
\begin{equation*}\label{eq:points}
	y=\eta-\sigma\in\mathcal{Y}_2, \qquad z=\eta-H_\nu\sigma\in F_y, 
\end{equation*}
with some $\eta\in S_{e_3}$ and $\sigma\in \Sigma_2\cap(S_{e_3}-y)$ such that the vectors $\eta$, $\sigma$, and $\nu$ are linearly independent, one can locally construct a $4\times4$ system of the form \autoref{eq:mateq}. In the following, we will keep the points $y$ and $z$ fixed and study the invertibility of the associated system matrix $B(w,\delta,\varepsilon)$. In particular, we derive a condition on the coefficient function $b\colon\Sigma_{2,\chi}\to\C$ that guarantees the determinant of the system matrix $B(w,\delta,\varepsilon)$ in \autoref{eq:mat} to be non-zero for at least one point $(w,\delta,\varepsilon)$ in every neighborhood of $0$.

The system in \autoref{eq:mat} fails to allow reconstruction in a neighborhood of $y$ only if there exists a neighborhood $U\subseteq\R^3$ of $0$ and an interval $J\subseteq\R$ around $0$ such that
\begin{equation*}
	0 = \det(B(w,\delta,\varepsilon))
	= b(\hat\sigma(w,0,\varepsilon))\,b(\hat\sigma(w,\delta,0))	- b(\hat\sigma(w,0,0))\,b(\hat\sigma(w,\delta,\varepsilon)),
\end{equation*}
that is,
\begin{equation}\label{eqB}
	\frac{b(\hat\sigma(w,\delta,\varepsilon))}{b(\hat\sigma(w,0,\varepsilon))}	= \frac{b(\hat\sigma(w,\delta,0))}{b(\hat\sigma(w,0,0))},
\end{equation}
for all $w\in U$ and $\delta,\varepsilon\in J$. 
This identity suggests working with the logarithm of $b$. Because $b(\sigma)\neq 0$ on $\Sigma_{2}$, we can on $\Sigma_{2,\chi}$ locally define
\[
c(\sigma)\coloneqq \log b(\sigma),
\]
(as around every $\sigma\in\Sigma_{2}$ there exists a neighborhood on which a smooth branch of the complex logarithm is well-defined).

	\begin{lemma}\label{thB}
We consider the \autoref{pr:main} for $d=3$ with the additional regularity assumption that $b\in C^2(\Sigma_{2,\chi};\C)$. Let $\eta\in S_{e_3}$ and $\sigma\in\Sigma_2$ be such that $\eta$, $\sigma$, and $\nu$ are linearly independent, and let the functions $\hat\eta$ and $\hat\sigma$ be defined as in \autoref{lem:pointsonline}.

	We assume that $b$ satisfies \autoref{eqB}. Then, with $c(\tilde\sigma)\coloneqq\log b(\tilde\sigma)$ for $\tilde\sigma\in\Sigma_{2,\chi}$ in a sufficiently small neighborhood of $\sigma$, we have
	\begin{equation}\label{eqC}
		\begin{split}
			&-\frac{1}{2\mu(\tilde\eta)}\mathrm D^2c(\sigma)\big(\tilde\eta\times\sigma,\tilde\eta\times\sigma\big)-\frac{\inner{\tilde\eta}\nu}{\mu(\tilde\eta)}\mathrm D^2c(\sigma)\big(\tilde\eta\times\sigma,\nu\times\sigma\big) \\
			&\qquad+\alpha(\tilde\eta)\mathrm Dc(\sigma)(\tilde\eta\times\sigma)+\beta(\tilde\eta)\mathrm Dc(\sigma)(\nu\times\sigma) = 0
		\end{split}
	\end{equation}
	for every $\tilde{\eta}\in S_{e_3}$ sufficiently close to $\eta$, where  the functions $\mu$, $\alpha$, and $\beta$ are defined in \autoref{eqMu}, \autoref{eqAlpha}, and \autoref{eqBeta}.
\end{lemma}
	\begin{proof}
	We assume that the domain $U\times J\times J$ of $\hat\eta$ and $\hat\sigma$ is chosen such that $c$ is well-defined at all points $\hat\sigma(w,\delta,\varepsilon)$ with $(w,\delta,\varepsilon)\in U\times J\times J$. Taking thus the logarithm of \autoref{eqB}, we find for fixed $(w,\delta)\in U\times J$ that the function
	\[
	\varepsilon\mapsto \log (b(\hat\sigma(w,\delta,\varepsilon)))- \log( b(\hat\sigma(w,0,\varepsilon)))
	\]
	is constant.
Calculating the derivative of this with respect to $\varepsilon$ at $\varepsilon=0$, we therefore get that
	\[ \mathrm Dc(\hat\sigma(w,\delta,0))\partial_\varepsilon\hat\sigma(w,\delta,0) = \mathrm Dc(\hat\sigma(w,0,0))\partial_\varepsilon\hat\sigma(w,0,0). \]
	This means that, for each fixed $w\in U$, the function
	\[ \delta\mapsto\mathrm Dc(\hat\sigma(w,\delta,0))\partial_\varepsilon\hat\sigma(w,\delta,0) \]
	is constant. Taking thus the derivative with respect to $\delta$ at $\delta = 0$ yields the relation
	\[ 0=\mathrm D^2c(\hat\sigma(w,0,0))\big(\partial_\varepsilon\hat\sigma(w,0,0),\partial_\delta\hat\sigma(w,0,0)\big)+\mathrm Dc(\hat\sigma(w,0,0))\partial_{\delta\varepsilon}\hat\sigma(w,0,0). \]
	
	Next, choose  $w=\tilde\eta-\eta$ for some $\tilde\eta\in S_{e_3}$ sufficiently close to $\eta$ such that $w\in U$. By \autoref{thConstantSol}, we then have $\hat\sigma(\tilde\eta-\eta,0,0)=\sigma$ and the relation simplifies to
	\[ \mathrm D^2c(\sigma)\big(\partial_\varepsilon\hat\sigma(\tilde\eta-\eta,0,0),\partial_\delta\hat\sigma(\tilde\eta-\eta,0,0)\big)+\mathrm Dc(\sigma)\partial_{\delta\varepsilon}\hat\sigma(\tilde\eta-\eta,0,0) = 0. \]
	
	Plugging in our expressions from \autoref{thSigmaSimple} for the derivatives of $\hat\sigma$, we obtain \autoref{eqC}.
\end{proof}

\begin{lemma}\label{thC}
We consider the \autoref{pr:main} for $d=3$ with the additional regularity assumption that $b\in C^2(\Sigma_{2,\chi};\C)$. Let $\eta\in S_{e_3}$ and $\sigma\in\Sigma_2$ be such that $\eta$, $\sigma$, and $\nu$ are linearly independent, and let the functions $\hat\eta$ and $\hat\sigma$ be defined as in \autoref{lem:pointsonline}.

If the function $b$ satisfies \autoref{eqB}, then
	\[ \mathrm Db(\sigma)(\nu\times\sigma) = 0. \]
\end{lemma}

\begin{proof}
	By \autoref{thB}, the logarithm $c(\tilde\sigma)\coloneqq \log b(\tilde\sigma)$ defined for $\tilde\sigma\in\Sigma_{2,\chi}$ sufficiently close to $\sigma$ satisfies \autoref{eqC}. Using the orthogonal basis $\sigma$, $\frac{\nu\times\sigma}{\norm{\nu\times\sigma}^2}$, $\frac{\sigma\times(\nu\times\sigma)}{\norm{\nu\times\sigma}^2}$, we represent the vector $\tilde{ \eta}\in S_{e_3}$ in \autoref{eqC} in the form
	\[ 
	\tilde\eta = \tilde\eta_1\sigma+\tilde\eta_2\frac{\nu\times\sigma}{\norm{\nu\times\sigma}^2}+\tilde\eta_3\frac{\sigma\times(\nu\times\sigma)}{\norm{\nu\times\sigma}^2}
	\]
	for suitable coefficients $\tilde{ \eta}_1$, $\tilde{ \eta}_2$, $\tilde{ \eta}_3$. 
	Substituting the expressions for the functions $\mu$, $\alpha$, and $\beta$ from \autoref{thCoeff} in \autoref{eqC} and using \autoref{eqEtaNu}, we then have 
	\begin{equation}\label{eq:expDgl}
	\begin{aligned}
		&\frac{\tilde\eta_2}2\mathrm D^2c(\sigma)\big(\tilde\eta\times\sigma,\tilde\eta\times\sigma\big)+\tilde\eta_2(\tilde\eta_1\inner\sigma\nu+\tilde\eta_3)\mathrm D^2c(\sigma)\big(\tilde\eta\times\sigma,\nu\times\sigma\big) \\
		&\qquad+\inner\sigma\nu\gamma(\tilde\eta)\mathrm Dc(\sigma)(\tilde\eta\times\sigma)+\big(\tilde\eta_2^2-k_0^2(\tilde\eta_1-1)\gamma(\tilde\eta)\big)\mathrm Dc(\sigma)(\nu\times\sigma) = 0.
	\end{aligned}
	\end{equation}
	A direct computation gives
	\[
	\tilde\eta\times\sigma = -\tilde\eta_2\frac{\sigma\times(\nu\times\sigma)}{\norm{\nu\times\sigma}^2}+k_0^2\tilde\eta_3\frac{\nu\times\sigma}{\norm{\nu\times\sigma}^2}. 
	\]
	 We now denote the directional derivatives of $c$ by
	\[ 
	C_1\coloneqq\mathrm Dc(\sigma)(\nu\times\sigma)\text{ and }C_2\coloneqq\mathrm Dc(\sigma)(\sigma\times(\nu\times\sigma))
	 \]
	and the corresponding second derivatives by
	\begin{align*}
		&A_{11}\coloneqq\mathrm D^2c(\sigma)\big(\nu\times\sigma,\nu\times\sigma\big), \quad A_{12}\coloneqq\mathrm D^2c(\sigma)\big(\nu\times\sigma,\sigma\times(\nu\times\sigma)\big), \text{ and} \\
		&A_{22}\coloneqq\mathrm D^2c(\sigma)\big(\sigma\times(\nu\times\sigma),\sigma\times(\nu\times\sigma)\big).
	\end{align*}
	
	Substituting these expressions into \autoref{eq:expDgl} yields a scalar relation of the form
	\begin{equation}\label{eqLinDep}
		C_1h_1(\tilde\eta)+C_2h_2(\tilde\eta)+A_{11}h_{11}(\tilde\eta)+A_{12}h_{12}(\tilde\eta)+A_{22}h_{22}(\tilde\eta) = 0,
	\end{equation}
	where 
	\begin{align*}
		&h_1(\tilde\eta)\coloneqq\tilde\eta_2^2-k_0^2(\tilde\eta_1-1)\gamma(\tilde\eta)+\frac{k_0^2\inner\sigma\nu}{\norm{\nu\times\sigma}^2}\gamma(\tilde\eta)\tilde\eta_3, \\
		&h_2(\tilde\eta)\coloneqq-\frac{\inner\sigma\nu}{\norm{\nu\times\sigma}^2}\gamma(\tilde\eta)\tilde\eta_2, \\
		&h_{11}(\tilde\eta)\coloneqq\frac{k_0^4\tilde\eta_2\tilde\eta_3^2}{2\norm{\nu\times\sigma}^4}-\frac{k_0^2\tilde\eta_2\tilde\eta_3}{\norm{\nu\times\sigma}^2}(\tilde\eta_1\inner\sigma\nu+\tilde\eta_3), \\
		&h_{12}(\tilde\eta)\coloneqq-\frac{k_0^2\tilde\eta_2^2\tilde\eta_3}{\norm{\nu\times\sigma}^4}-\frac{\tilde\eta_2^2}{\norm{\nu\times\sigma}^2}(\tilde\eta_1\inner\sigma\nu+\tilde\eta_3),\text{ and} \\
		&h_{22}(\tilde\eta)\coloneqq\frac{\tilde\eta_2^3}{2\norm{\nu\times\sigma}^4}.
	\end{align*}
	
	Consider $\tilde\eta_3=\sqrt{k_0^2-\tilde\eta_1^2-\tilde\eta_2^2}$ as a function of $\tilde\eta_1$ and $\tilde\eta_2$. Then, the left-hand side of \autoref{eqLinDep} is a linear combination of five holomorphic functions in $\tilde\eta_1$ and $\tilde\eta_2$. If this equation has to hold in a neighborhood of $\eta$, it must hold for all $\tilde\eta\in S_{e_3}.$
	
	Setting $\tilde{\eta}_2=0$, we see that $h_2$, $h_{11}$, $h_{12}$, and $h_{22}$ are zero, while $h_{1}$ does not vanish. 
	Since \autoref{eqLinDep} must hold for all $\tilde{\eta}$, the remaining coefficient must therefore satisfy $C_1 =0 $
	which is equivalent to $ Db(\sigma)(\nu\times\sigma) =0$. 
\end{proof}

\begin{proposition}\label{prop:uq3}
We consider the \autoref{pr:main} for $d=3$ with the additional regularity assumption that $b\in C^2(\Sigma_{2,\chi};\C)$. Let $\eta\in S_{e_3}$ and $\sigma\in\Sigma_2$ be such that $\eta$, $\sigma$, and $\nu$ are linearly independent.

If $g\in C(\R^d;\C)$ is a continuous solution of \autoref{eq:sys_eq} and $b$ fulfills
\[\mathrm Db(\sigma)(\nu\times\sigma) \neq 0,\]
then $g(\eta-\sigma)=0$. 
\end{proposition}

\begin{proof}
We set $y\coloneqq\eta-\sigma$ and $z\coloneqq\eta-H_\nu\sigma$. By \autoref{cor:closed}, there exist neighborhoods $U$ of the origin and $J$ of $0$ such that for every $(w,\delta,\varepsilon) \in U \times J \times J$ the function $g$  satisfies
	\begin{align*}
		B \begin{pmatrix}
			g(y+w)\\g(y+w+\delta\nu)\\g(z+w)\\g(z+w+\varepsilon\nu)
		\end{pmatrix} =0, 
	\end{align*}
	with the coefficient matrix $B= B(w,\delta,\varepsilon)$ as defined in \autoref{eq:mat}.

	By assumption, we have $\mathrm Db(\sigma)(\nu\times\sigma)\neq 0$. Combining the results  \autoref{thC} and \autoref{thB} we conclude that $(w,\delta,\varepsilon)\mapsto \det B(w,\delta,\varepsilon)$ is not identically zero in any neighborhood $\tilde U\times\tilde J\times\tilde J\subseteq U\times J\times J$ of $(0,0,0)$. Hence there exists at least one triple  $(w',\delta',\varepsilon')$ with $\det B(w',\delta',\varepsilon')\neq 0$ and therefore $g(y+w') = 0$.
	
	Now, choose a sequence $(w_n,\delta_n,\varepsilon_n)\to (0,0,0)$ with $\det B_n (w_n,\delta_n,\varepsilon_n) \neq 0$. From the previous step we know $g(y+w_n) =0$ for all $n\in\N$. Since $g$ is continuous, letting $n\to\infty$, gives
	\[
	g(y) = \lim\limits_{n\to \infty}g(y+w_n)=0,
	\]
	which completes the proof.
\end{proof}

\begin{remark}\label{rem:lindep}
	For $y\in\mathcal{Y}_2\cap\mathcal{B}^3_{2k_0}\setminus\R\nu$ there exist infinitely many representations $y = \eta-\sigma$ with $\sigma\in\Sigma_2\cap(S_{e_3}-y)$ and $\eta \in S_{e_3}$. (For points $\norm{y}=2k_0$, the intersection $\Sigma_2\cap(S_{e_3}-y)$ degenerates to a single point.) For such a representation, the vectors  $\eta$, $\sigma$, and $\nu$ are linearly independent if and only if
		\[
		\inner{\sigma}{\nu\times y} \ne  0,
		\]	
	since $\eta = y+\sigma$.  Equivalently, linear dependence occurs when $\sigma$ lies in the plane $\operatorname{span}\{y,\nu\}$. However, every $\sigma$ representing $y$ must lie in $\Sigma_2\cap(S_{e_3}-y)$, which is a circular arc lying in a plane orthogonal to $y$. This arc intersects the plane $\operatorname{span}\{y,\nu\}$ in at most two points. Therefore, there exist a whole continuum of admissible $\sigma$ for which $\eta$, $\sigma$, and $\nu$ are linearly independent. 
\end{remark}

\begin{theorem}\label{thm:uq3}	
	We consider the \autoref{pr:main} for $d=3$ with the additional assumptions that $b\in C^2(\Sigma_{2,\chi}; \C)$ and that
	\[
	\overline{\left\lbrace \sigma\in\Sigma_2\mid \nabla b(\sigma)\notin\C\pi_\sigma\nu\right\rbrace} = \Sigma_2,
	\]
	where $\pi_\sigma$ denotes the orthogonal projection onto $\sigma^\perp$ as defined in \autoref{eq:projection}.
	
	Every continuous solution $g\in C(\R^3;\C)$ of \autoref{eq:sys_eq} then fulfills $g(y) = 0 $ for all $y\in\mathcal{Y}_1\cup\mathcal{Y}_2$.
\end{theorem}

\begin{proof}
	The statement on $\mathcal{Y}_1$ follows from \autoref{lem:uqS1}.
	Let $y\in\mathcal{Y}_2$ be arbitrary and written as $y = \eta - \sigma$ for some $\eta\in S_{e_3}$ and $\sigma\in \Sigma_2$. 
	Since $\nabla b(\tau)\in\C\pi_\tau\nu$ is (because of $\mathrm Db(\tau)\tau=0$ due to $b(\tilde\chi\tau)=b(\tau)$ for all $\tilde\chi\in(1-\chi,1+\chi)$) equivalent to $\mathrm Db(\tau)(\nu\times\tau)=0$, the set
	\[
	G \coloneqq \{\tau\in\Sigma_2 \mid \mathrm Db(\tau)(\nu\times\tau)\neq 0\}
	\]
	is by assumption dense in $\Sigma_2$.  Therefore, we can construct a sequence  
	$(\tau_n)_{n\in\N}\subseteq G$ with $\tau_n\to\tau$ as $n\to\infty$, where we have for each $n$
	\[
	\mathrm Db(\tau_n)(\nu\times\tau_n)\neq 0.
	\]
	According to \autoref{rem:lindep}, we can select this sequence such that in addition $\eta$, $\tau_n$, and $\nu$  are linearly independent for all $n\in\N$.
	Hence, for $y_n\coloneqq\eta-\tau_n$ all assumptions of \autoref{prop:uq3} are satisfied at, which yields
	\[
	g(y_n)=0 \qquad \text{for all } n\in\N.
	\]
	By the continuity of $g$ we conclude
	\[
	g(y)=\lim_{n\to\infty} g(y_n) = 0.
	\]
	
\end{proof}

\section{Reconstructible Fourier coefficients in two dimensions}\label{sec:two}
Lastly, we turn to the analysis of \autoref{pr:main} the two-dimensional setting. Recall that, in this case, we introduced a graph on $\mathcal{Y}_2$, where each equation of the form \autoref{eq:homFy} corresponds to an edge connecting two vertices, see \autoref{def:graph}.

We ignore all degenerate points as they form a set of measure zero and restrict our attention to the set $\mathcal Y_2'\subseteq\mathcal Y_2$ of non-degenerate points introduced in \autoref{def:nondegY2}. The structural classification of the connected components of the associated graph, established in \autoref{lem:Zhk}, reveals that each such non-degenerate point belongs to an equation system with at most four equations involving function values of $g$ at not more than four points.

Whenever a point $y\in \mathcal{Y}_2'$ lies in a component containing fewer edges than vertices, then the associated system is underdetermined. In such cases, uniqueness of a solution cannot be expected without incorporating additional information. 

As a consequence, there remain, according to \autoref{lem:Zhk}, precisely two approaches by which uniqueness may still arise for the vertices of a connected component of a vertex in $\mathcal Y_2'$:
\begin{enumerate}
	\item[a)] A vertex $y$ also lies in $\mathcal{Y}_1$, where \autoref{lem:uqS1} ensures $g(y)=0$, implying that the function is  zero on all vertices in that component.
	\item[b)] The connected component contains four vertices and four edges such that the corresponding equation system consists of four equations for the four values at the vertices. In this case one must determine whether this $4\times4$ system admits a unique solution.
\end{enumerate}

In what follows, we examine the cases a) and b) in detail.

\subsection{Identification of graph vertices in $\mathcal{Y}_1$}
We begin with case a) which requires identifying those vertices of the graph that lie in the intersection $\mathcal{Y}_1\cap\mathcal{Y}_2$. 
 Since an element $y$ of this intersection is of the form $y=\eta_1-\sigma_1=\eta_2-\sigma_2$ for some $\eta_1,\eta_2\in S_{e_2}$, $\sigma_1\in\Sigma_1$, and $\sigma_2\in\Sigma_2$, it lies in the intersection of the two semicirles $S_{e_2}-\sigma_1$ and $S_{e_2}-\sigma_2$. Let us therefore briefly describe this set.

\begin{lemma}\label{lem:int_circles}
	Let $d=2$ and $\sigma_1, \sigma_2 \in \mathbb{S}_{k_0}^1$ with $\sigma_1\neq\sigma_2$ and consider two upper semicircles $S_{e_2}+\sigma_1$ and $S_{{e}_2}+\sigma_2$ centered at $\sigma_1$ and $\sigma_2$, respectively.  
	Then, the intersection is given by
	\[
	(S_{{e}_2}+\sigma_1) \cap (S_{{e}_2}+\sigma_2) =
	\begin{cases}
		\left\lbrace \sigma_1 + \sigma_2 \right\rbrace & \text{if } \sigma_1 , \sigma_2 \in S_{{e}_2}, \\
		\left\lbrace 0 \right\rbrace  & \text{if }  \sigma_1 , \sigma_2 \in {S}_{-{e}_2},\\
		\emptyset & \text{otherwise}.
	\end{cases}
	\]
\end{lemma}
\begin{proof}
	Let $y \in (S_{{e}_2}+\sigma_1) \cap (S_{{e}_2}+\sigma_2)$ and $\sigma_1, \sigma_2 \in \mathbb{S}_{k_0}^1$ with $\sigma_1\neq\sigma_2$. That is, we have $\norm{y-\sigma_1}^2 = 1$ and $\norm{y-\sigma_2}^2 = 1$ so that
	\begin{align*}
		\norm{y}^2 - 2\inner{y}{\sigma_1} + 1 = \norm{y}^2 - 2\inner{y}{\sigma_2} + 1, \quad \text{which is equivalent to} \quad \inner{y}{\sigma_1-\sigma_2} = 0.
	\end{align*}
	The only two points that satisfy both this orthogonality constraint and the condition $\norm{y-\sigma_1}=1$ are the origin, $y=0$, and the sum of the centers, $y=\sigma_1+\sigma_2$. 
	
	We check when these two points satisfy the required semicircle conditions, namely $y-\sigma_1 \in S_{e_2}$ and $y-\sigma_2 \in S_{e_2}$.
	The point $y=0$ satisfies these conditions if and only if $-\sigma_1, -\sigma_2 \in S_{e_2}$, which is equivalent to $\sigma_1, \sigma_2 \in S_{-e_2}$.
	The point $y=\sigma_1+\sigma_2$ satisfies these conditions if and only if $\sigma_1,\sigma_2 \in S_{e_2}$.
	
	If neither of these two distinct conditions is met, the intersection is empty. 
\end{proof}

\begin{figure}[t]
	\begin{subfigure}{0.45\textwidth}
		\centering
		\includegraphics{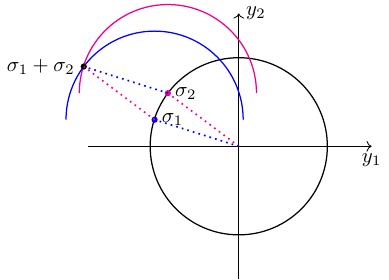}
		\caption{$\sigma_1, \sigma_2 \in S_{e_2}$ }
		\label{subfig:uppercirc}
	\end{subfigure}\hfill
	\begin{subfigure}{0.45\textwidth}
		\includegraphics{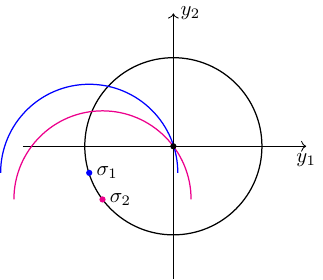}
		\caption{$\sigma_1, \sigma_2 \in S_{-e_2}$}
		\label{subfig:lowercirc}
	\end{subfigure}
	\caption{Illustration of the intersection $(\sigma_1+S_{e_2})\cap (\sigma_2+S_{e_2})$ for different positions of the centers. If both centers lie on the upper semicircle, the intersection is given by their sum; otherwise, the intersection lies in the origin.}
	\label{fig:int_circles}
\end{figure}

See \autoref{fig:int_circles} for an illustration of the intersections between two upper semicircles corresponding to different center positions.

\begin{lemma}\label{lem:intY}
	We consider the \autoref{pr:main} for $d=2$. The intersection of the subsets $\mathcal{Y}_1$ and $\mathcal{Y}_2$ then fulfills
	\[
	(\mathcal{Y}_1\cap \mathcal{Y}_2)\setminus\{0\} = \left\lbrace \eta -\sigma \mid \eta\in (-\Sigma_1)\cap S_{e_2}, \ \sigma \in\Sigma_2\cap S_{-e_2}\right\rbrace.
	\]
\end{lemma}
\begin{proof}
We get by definition of the sets $\mathcal Y_1$ and $\mathcal Y_2$ in \autoref{eq:Y1_Y2} that
\[ \mathcal Y_1\cap\mathcal Y_2 = \bigcup_{\sigma_1\in\Sigma_1}\bigcup_{\sigma_2\in\Sigma_2}\big((S_{e_2}-\sigma_1)\cap(S_{e_2}-\sigma_2)\big). \]
Plugging in the expression for the intersection from \autoref{lem:int_circles}, we obtain
\[ \mathcal Y_1\cap\mathcal Y_2 = \begin{cases}
\{-\sigma_1-\sigma_2\mid\sigma_1\in\Sigma_1\cap S_{-e_2},\,\sigma_2\in\Sigma_2\cap S_{-e_2}\} &\text{if }\Sigma_1\cap S_{e_2}=\emptyset\text{ or }\Sigma_2\cap S_{e_2}=\emptyset, \\
\{-\sigma_1-\sigma_2\mid\sigma_1\in\Sigma_1\cap S_{-e_2},\,\sigma_2\in\Sigma_2\cap S_{-e_2}\}\cup\{0\}&\text{otherwise}.
\end{cases} \]
\end{proof}
Having described the intersection $\mathcal{Y}_1\cap\mathcal{Y}_2$, we can now characterize which vertices of the graph $(\mathcal{Y}_2, E)$ lie within $\mathcal{Y}_1$. 

\begin{proposition}\label{prop:vertexY1}
We consider the \autoref{pr:main} for $d=2$. Let $(\mathcal{Y}_2,E)$ be the graph associated with the equation system in \autoref{eq:homFy} and let $(C,\ E_C)$  be a connected component of this graph containing a non-degenerate point from $\mathcal Y_2'$. 

If we have $C\cap\mathcal{Y}_1\ne\emptyset$, then $(C,\ E_C)$ is of the form
\begin{equation}\label{eq:CY1}
\parbox[t]{0.3\textwidth}{\mbox{}\\[-2ex]\includegraphics{connectedComponents_1.pdf}}\text{ with }\eta\in (-\Sigma_1)\cap S_{e_2}\text{ and }\sigma\in\Sigma_2\cap S_{-e_2}.
\end{equation}
And we get that $C\subseteq\mathcal Y_1$ if and only if we additionally have $H_\nu\sigma\in S_{-e_2}$.
\end{proposition}

\begin{proof}
Let $y\in C\cap\mathcal Y_1$. According to \autoref{lem:intY}, we can then find $\eta\in(-\Sigma_1)\cap S_{e_2}$ and $\sigma\in\Sigma_2\cap S_{-e_2}$ with $y=\eta-\sigma$.

By \autoref{lem:Zhk} we know that every connected component including a point from $\mathcal Y_2'$ with more than two vertices consists only of points of the form $\tilde\eta-\tilde\sigma$ with $\tilde\eta\in(-\Sigma_2)$ and $\tilde\sigma\in\Sigma_2$. Since $y\in\mathcal{Y}_1$, it cannot belong to such a component. Hence, the connected component $C$ must consist of exactly two vertices and we get \autoref{eq:CY1}.

Moreover, we have again by \autoref{lem:intY} that $\eta-H_\nu\sigma\in\mathcal Y_1$ is equivalent to $\eta\in(-\Sigma_1)\cap S_{e_2}$ and $H_\nu\sigma\in\Sigma_2\cap S_{-e_2}$. Since the condition on $\eta$ is already satisfied if we impose $\eta-\sigma\in\mathcal Y_1$, we obtain that $C\subseteq\mathcal{Y}_1$ if and only if $H_\nu\sigma\in S_{-e_2}$.
\end{proof}

\autoref{prop:vertexY1} yields the following conclusions:
first, for those connected components $C$ which are neither of the form as in \autoref{eq:CY1} nor a subset of the degenerate point set $\mathcal Y_2\setminus\mathcal Y_2'$, the system associated with that component is indeed underdetermined. In this case, there therefore exists for every $y\in C$ a solution $g\colon \R^2\to\C$ of \autoref{eq:sys_eq} with $g(y)\neq 0$.  Second, we can identify a subset $\tilde{\mathcal{Y}}\subseteq \mathcal{Y}_2\setminus\mathcal{Y}_1$ on which the system admits only the trivial solution. More precisely,
\begin{equation}\label{eq:tildeY}
\tilde{\mathcal{Y}}\coloneqq \left\lbrace \eta -H_\nu\sigma\mid \eta\in (-\Sigma_1)\cap S_{e_2},\ \sigma\in\Sigma_2\cap S_{-e_2},\ H_\nu\sigma\notin S_{-e_2} \right\rbrace.
\end{equation}

\subsection{Connected components with four edges}
We now proceed with case b). Here we consider those points in $\mathcal{Y}_2'$ that belong to connected components containing four vertices and four edges.  These components form a closed subsystem of \autoref{eq:homFy}, which, by \autoref{prop:vertexY1}, are completely isolated from points in $\mathcal{Y}_1$. We will now show that, despite being closed, these subsystems are still underdetermined and therefore admit non-trivial solutions.
\begin{proposition}\label{thm:circle_system}
We consider the \autoref{pr:main} for $d=2$. Let $(\mathcal{Y}_2,E)$ be the graph associated with \autoref{eq:homFy}. Consider a connected component $(C,\ E_C)$ of this graph containing a vertex from the set $\mathcal Y_2'$ of non-degenerate points of $\mathcal Y_2$ of the form 
\[ \parbox[t]{0.3\textwidth}{\mbox{}\\[-2ex]\includegraphics{connectedComponents_4.pdf}} \quad\text{for some }\eta\in(-\Sigma_2)\cap S_{e_2}\text{ and }\sigma\in\Sigma_2\cap S_{-e_2}. \]
Then, for every $y\in C$, there exists a solution $g\colon \R^2\to \C$ of \autoref{eq:sys_eq} with $g(y)\neq 0$.
\end{proposition}

\begin{proof}

The equations associated with the vertices in $C$ can be written as the linear system
	\begin{align*}
		B_{\eta,\sigma}
		\begin{pmatrix}
			g(\eta-\sigma)\vspace*{1mm}\\ 
			g(\eta -H_\nu\sigma )\vspace*{1mm}\\
			g(H_\nu \eta- H_\nu\sigma)\vspace*{1mm}\\
			g(H_\nu\eta-\sigma) 
		\end{pmatrix} = 0
		 \quad \text{with}\quad 	B_{\eta,\sigma}\coloneqq \begin{pmatrix}
		b\left(\sigma\right) & 1 &0 &0\vspace*{1mm} \\
		b\left( -\eta\right) & 0& 0& 1\vspace*{1mm} \\
		0 & 0& 1& b\left( \sigma\right)\vspace*{1mm} \\
		0& b\left( -\eta\right)& 1 & 0
		\end{pmatrix}.
	\end{align*}
	 A direct computation shows that the kernel of $B_{\eta,\sigma}$ is one-dimensional and given by 
	 \[
	 \ker B_{\eta,\sigma} = \operatorname{span}\left\lbrace \begin{pmatrix}
	 	1\\ -b(\sigma) \\ b(\sigma)b(-\eta)\\-b(-\eta)
	 \end{pmatrix}\right\rbrace. 
	 \]
	 In particular, we can find a solution $g\colon\R^2\to\C$ of \autoref{eq:sys_eq} with $g(y)\neq 0$ for all $y\in C$.
\end{proof}

Finally, we combine \autoref{lem:uqS1}, \autoref{prop:vertexY1}, and \autoref{thm:circle_system} to answer the question of reconstructibility of the Fourier coefficients on $\mathcal{Y}_2$ in two dimensions. For simplicity, we exclude the degenerate points in $\mathcal Y_2\setminus\mathcal Y_2'$, since these sets have measure zero. Accordingly, the result is formulated almost everywhere.
\begin{theorem}\label{thm:uq1}
We consider the \autoref{pr:main} for $d=2$.	Let $\tilde{\mathcal{Y}}$ be defined as in \autoref{eq:tildeY}.
\begin{enumerate}
\item
Then, we have for every solution $g\colon\R^2\to \C$ of \autoref{eq:sys_eq} that
\[
g=0 \quad \text{almost everywhere on} \ \mathcal{Y}_1\cup \tilde{\mathcal{Y}}.
\]
\item
Conversely, there exists for almost every $y\in\mathcal{Y}_2\setminus(\mathcal{Y}_1\cup \tilde{\mathcal{Y}})$ a solution $g\colon\R^2\to \C$ of \autoref{eq:sys_eq} with $g(y)\neq 0$.
\end{enumerate}
\end{theorem}

Thus, in contrast to the higher-dimensional case, the unique reconstructibility of the Fourier transform of the scattering potential on $\mathcal{Y}_2$ fails in two dimensions, and $\tilde{\mathcal{Y}}$ is the largest region outside $\mathcal{Y}_1$ on which only the trivial solution of \autoref{eq:sys_eq} exists. In \cite{ElbNauSch26_preprint} the resulting regions $\mathcal Y_1$, $\mathcal{Y}_2$, and $\tilde{\mathcal{Y}}$ are visualized for different scanning geometries.

\section{Conclusion}
In this article, we analyzed the Fourier diffraction relation arising in the extension of diffraction tomography to general beam scanning geometries, as introduced in \cite{ElbNauSch26_preprint}. In this setting, the Fourier diffraction relation may couple two Fourier coefficients of the objects scattering potential through a single measurement, which raises the question of whether these coefficients can be uniquely recovered.

For dimensions $d \geq 3$, we proved that all Fourier coefficients appearing in the measurements can, under mild assumptions on the Herglotz density $a \colon \Sp_{k_0}^{d-1} \to \mathbb{C}$, be uniquely determined from the data. In two dimensions, however, only a subset of these Fourier coefficients is uniquely recoverable, while the remaining coefficients cannot be reconstructed from the measurements.

Nevertheless, even in that case, the recoverable Fourier coefficients go beyond those that can be directly read off from the Fourier data of the measurements.

	\subsection*{Acknowledgments}
	%
	%
	This research was funded in whole, or in part, by the Austrian Science Fund
	(FWF) SFB 10.55776/F68 ``Tomography Across the Scales'', project F6804-N36 (Quantitative Coupled Physics Imaging). For open access purposes, the author has
	applied a CC BY public copyright license to any author-accepted manuscript
	version arising from this submission.
	The financial support by the Austrian Federal Ministry for Digital and Economic
	Affairs, the National Foundation for Research, Technology and Development and the Christian Doppler
	Research Association is gratefully acknowledged.
	\appendix
	\section{Appendix}
	In \autoref{subsec:3d} we considered \autoref{pr:main} for $d=3$ and derived a condition on the coefficient function $b\in C^2(\Sigma_{2,\chi};\C)$ ensuring that the determinant of the system matrix in \autoref{eq:mat} does not vanish.  In this context, we introduced for fixed points $\eta\in S_{e_3}$ and $\sigma\in\Sigma_2$ for which the vectors $\eta$, $\sigma$, and $\nu$ are linearly independent on sufficiently small neighborhoods $U\subseteq\R^3$ of the origin and $J\subseteq\R$ of $0$, the smooth maps
	\[ \hat\eta\colon U\times J\times J\to S_{e_3}\text{ and }\hat\sigma\colon U\times J\times J\to\Sigma_2, \]
	defined by $\hat\eta(0,0,0)=\eta$, $\hat\sigma(0,0,0)=\sigma$, and
	\begin{equation}\label{eq:defeq}
		y+w+\delta\nu = \hat\eta(w,\delta,\varepsilon)-\hat\sigma(w,\delta,\varepsilon),\qquad 
		z+w+\varepsilon\nu=\hat\eta(w,\delta,\varepsilon)-H_\nu \hat\sigma(w,\delta,\varepsilon)
	\end{equation}
	where $y\coloneqq\eta-\sigma$ and $z\coloneqq\eta-H_\nu\sigma$. 
	
	In this appendix, we collect the auxiliary material used in the derivation of the determinant condition presented in \autoref{thB} and \autoref{thC}. In particular, we perform the computation of the first and second derivatives of $\hat\sigma$, which enter the expressions in these results.
	
	In the following, we keep the notation introduced above for all lemmata and we assume that the neighborhoods $U$ and $J$ are sufficiently small such that the vectors $\nu$, $y+w+\delta\nu=\hat\eta(w,\delta,\varepsilon)-\hat\sigma(w,\delta,\varepsilon)$, and $\hat\sigma(w,\delta,\varepsilon)$ form a basis for all $w\in U$ and $\delta,\varepsilon\in J$.

	\begin{lemma}\label{thConstantSol}
		We have for all $\tilde\eta\in S_{e_3}$ with $-\eta+\tilde\eta\in U$ that
		\begin{align*}
			&\hat\sigma(-\eta+\tilde\eta,0,0)=\sigma\text{ and} \\
			&\hat\eta(-\eta+\tilde\eta,0,0)=\tilde\eta.
		\end{align*}
	\end{lemma}
	\begin{proof}
		By definition, the values $\hat\eta(-\eta+\tilde\eta,0,0)$ and $\hat\sigma(-\eta+\tilde\eta,0,0)$ are given as the solution of
		\begin{align*}
			&y-\eta+\tilde\eta = \hat\eta(-\eta+\tilde\eta,0,0)-  \hat\sigma(-\eta+\tilde\eta,0,0)\text{ and} \\
			&z-\eta+\tilde\eta = \hat\eta(-\eta+\tilde\eta,0,0) - H_\nu\hat\sigma(-\eta+\tilde\eta,0,0)
		\end{align*}
		with $\hat\eta(0,0,0)=\eta$ and $\hat\sigma(0,0,0)=\sigma$.
		
		Choosing $\hat\eta(-\eta+\tilde\eta,0,0)=\tilde\eta$, we see that this reduces to
		\begin{align*}
			&y = \eta-\hat\sigma(-\eta+\tilde\eta,0,0)\text{ and} \\
			&z=\eta-H_\nu\hat\sigma(-\eta+\tilde\eta,0,0),
		\end{align*}
		which has $\hat\sigma(-\eta+\tilde\eta,0,0)=\sigma$ as a solution.
	\end{proof}
	
	\begin{lemma}\label{thSigma}

		The derivatives of $\hat\sigma$ are given by
		\begin{align*}
			&\partial_\varepsilon\hat\sigma(w,\delta,\varepsilon) = \frac12v_1(w,\delta,\varepsilon), \\
			&\partial_\delta\hat\sigma(w,\delta,\varepsilon) = -\frac12v_1(w,\delta,\varepsilon)-\inner{\hat\eta(w,\delta,\varepsilon)}\nu v_2(w,\delta,\varepsilon),\text{ and} \\
			&\partial_{\delta\varepsilon}\hat\sigma(w,\delta,\varepsilon) = -\frac12v_2(w,\delta,\varepsilon)+\left(\frac14\norm{v_1(w,\delta,\varepsilon)}^2+\frac12\inner{\hat\eta(w,\delta,\varepsilon)}\nu\inner{v_1(w,\delta,\varepsilon)}{v_2(w,\delta,\varepsilon)}\right)v_3(w,\delta,\varepsilon),
		\end{align*}
		where $(v_j(w,\delta,\varepsilon))_{j=1}^3$ denotes the dual basis to $\nu$, $\hat\eta(w,\delta,\varepsilon)-\hat\sigma(w,\delta,\varepsilon)$, and $\hat\sigma(w,\delta,\varepsilon)$.
	\end{lemma}

\begin{proof}
 We first calculate the projections of the derivatives of $\hat\sigma$ to the basis vectors $\nu$, $y+w+\delta\nu$, and $\hat\sigma(w,\delta,\varepsilon)$. For this purpose, we eliminate the unknown $\hat\eta$ from \autoref{eq:defeq} and find the condition
	\[ y-z+(\delta-\varepsilon)\nu = -\hat\sigma(w,\delta,\varepsilon)+H_\nu\hat\sigma(w,\delta,\varepsilon) = -2\inner{\hat\sigma(w,\delta,\varepsilon)}\nu\nu. \]
	Using that $y-z=-2\inner\sigma\nu\nu$, this implies that
	\[ 
	\inner{\hat\sigma(w,\delta,\varepsilon)}\nu = \inner\sigma\nu-\frac12(\delta-\varepsilon). 
	\]
	We thus have
	\begin{equation}\label{eq:ProjectionsSigmaBasis1}
	\langle \partial_\varepsilon \hat\sigma(w,\delta,\varepsilon), \nu \rangle = \frac12, \quad
	\langle \partial_\delta \hat\sigma(w,\delta,\varepsilon), \nu \rangle = -\frac12, \quad
	\langle \partial_{\delta\varepsilon} \hat\sigma(w,\delta,\varepsilon), \nu \rangle = 0.
	\end{equation}
	
	Additionally, we know that
	\[ \hat\sigma(w,\delta,\varepsilon)\in\Sigma_2\cap(S_{e_3}-(y+w+\delta\nu))\subseteq\S_{k_0}^2\cap(\S_{k_0}^2-(y+w+\delta\nu)) \]
	so that it lies in particular in the plane through $-\frac{y+w+\delta\nu}2$ orthogonal to $y+w+\delta\nu$, that is,
	\[ \inner{\hat\sigma(w,\delta,\varepsilon)+\tfrac12(y+w+\delta\nu)}{y+w+\delta\nu} = 0. \]
	We take the derivative of this with respect to $\delta$ and get
	\[ \inner{\partial_\delta\hat\sigma(w,\delta,\varepsilon)+\tfrac12\nu}{y+w+\delta\nu}+\inner{\hat\sigma(w,\delta,\varepsilon)+\tfrac12(y+w+\delta\nu)}{\nu} = 0. \]
	
	Taking also the derivatives of the last two equations with respect to $\varepsilon$ gives us the expressions
	\begin{equation}\label{eq:ProjectionsSigmaBasis2}
	\begin{aligned}
			&\inner{\partial_\varepsilon\hat\sigma(w,\delta,\varepsilon)}{y+w+\delta\nu} = 0, \\
			&\inner{\partial_\delta\hat\sigma(w,\delta,\varepsilon)}{y+w+\delta\nu} =- \inner{y+w+\delta\nu+\hat\sigma(w,\delta,\varepsilon)}\nu = -\inner{\hat\eta(w,\delta,\varepsilon)}\nu,\\
			&\inner{\partial_{\delta\varepsilon}\hat\sigma(w,\delta,\varepsilon)}{y+w+\delta\nu} = -\inner{\partial_\varepsilon\hat\sigma(w,\delta,\varepsilon)}\nu = -\frac12.
		\end{aligned}
	\end{equation}
	
	Moreover, we have $\|\hat\sigma(w,\delta,\varepsilon)\|^2 = \inner{\hat \sigma(w,\delta,\varepsilon)}{\hat{\sigma}(w,\delta,\varepsilon)} = k_0^2$ and therefore
	\begin{equation}\label{eq:ProjectionsSigmaBasis3}
	\begin{aligned}
		&\langle \partial_\varepsilon \hat\sigma(w,\delta,\varepsilon), \hat\sigma (w,\delta,\varepsilon)\rangle =0, \\ 
		&\langle \partial_\delta \hat\sigma(w,\delta,\varepsilon), \hat\sigma(w,\delta,\varepsilon) \rangle = 0,\\
		&\langle \partial_{\delta\varepsilon} \hat\sigma(w,\delta,\varepsilon), \hat\sigma(w,\delta,\varepsilon) \rangle = - \langle \partial_\varepsilon \hat\sigma(w,\delta,\varepsilon), \partial_\delta \hat\sigma (w,\delta,\varepsilon)\rangle.
		\end{aligned}
	\end{equation}

	The partial derivatives of $\hat\sigma$ can now be expressed as a linear combination of the dual basis $(v_j(w,\delta,\varepsilon))_{j=1}^3$ to $\nu$, $y+w+\delta\nu =\hat\eta(w,\delta,\varepsilon)-\hat\sigma(w,\delta,\varepsilon)$, and $\hat\sigma(w,\delta,\varepsilon)$. We get with \autoref{eq:ProjectionsSigmaBasis1}, \autoref{eq:ProjectionsSigmaBasis2}, and \autoref{eq:ProjectionsSigmaBasis3} that
	\begin{align*}
		&\partial_\varepsilon \hat\sigma(w,\delta,\varepsilon) = \frac12 v_1(w,\delta,\varepsilon), \\
		&\partial_\delta \hat\sigma (w,\delta,\varepsilon)= -\frac12 v_1(w,\delta,\varepsilon) - \langle \hat\eta(w,\delta,\varepsilon), \nu \rangle v_2(w,\delta,\varepsilon).
	\end{align*}
	Plugging this into the last equation from \autoref{eq:ProjectionsSigmaBasis3} gives us for the second derivative
	\[ 
		\inner{\partial_{\delta\varepsilon}\hat\sigma(w,\delta,\varepsilon)}{\hat\sigma(w,\delta,\varepsilon)} = \frac14\norm{v_1(w,\delta,\varepsilon)}^2+\frac12\inner{\hat\eta(w,\delta,\varepsilon)}\nu\inner{v_1(w,\delta,\varepsilon)}{v_2(w,\delta,\varepsilon)},
	 \]
	and therefore
	\[
		\partial_{\delta\varepsilon} \hat\sigma(w,\delta,\varepsilon) = -\frac12 v_2(w,\delta,\varepsilon) +\Big(  \frac14 \|v_1(w,\delta,\varepsilon)\|^2 + \frac12 \langle \hat\eta(w,\delta,\varepsilon), \nu \rangle \langle v_1(w,\delta,\varepsilon), v_2(w,\delta,\varepsilon) \rangle \Big) v_3(w,\delta,\varepsilon).
	\]
\end{proof}
	
	\begin{lemma}\label{thSigmaSimple}
		The derivatives of $\hat\sigma$ at points $(-\eta+\tilde\eta,0,0)$ for vectors $\tilde\eta\in S_{e_3}$ such that $-\eta+\tilde\eta\in U$ are given by
		\begin{align}
			&\partial_\varepsilon\hat\sigma(-\eta+\tilde\eta,0,0) = \frac12\mu(\tilde\eta)\tilde\eta\times\sigma, \label{eqSigmaEps} \\
			&\partial_\delta\hat\sigma(-\eta+\tilde\eta,0,0) = -\frac12\mu(\tilde\eta)\tilde\eta\times\sigma-\mu(\tilde\eta)\inner{\tilde\eta}\nu\nu\times\sigma,\text{ and} \label{eqSigmaDelta} \\
			&\pi_\sigma\big(\partial_{\delta\varepsilon}\hat\sigma(-\eta+\tilde\eta,0,0)\big) = \frac12\mu^3(\tilde\eta)\big(\alpha(\tilde\eta)\tilde\eta\times\sigma+\beta(\tilde\eta)\nu\times\sigma\big), \label{eqSigmaDeltaEps}
		\end{align}
		with the functions
		\begin{align}
			&\mu(\tilde\eta)\coloneqq\frac1{\inner{\nu\times\tilde\eta}\sigma}, \label{eqMu} \\
			&\alpha(\tilde\eta) \coloneqq \inner\sigma\nu\gamma(\tilde\eta), \label{eqAlpha} \\
			&\beta(\tilde\eta) \coloneqq \frac1{\mu^2(\tilde\eta)}-(\inner\sigma{\tilde\eta}-k_0^2)\gamma(\tilde\eta),\text{ and} \label{eqBeta} \\
			&\gamma(\tilde\eta) \coloneqq \frac1{k_0^2}\left(\inner{\tilde\eta}\nu\inner{\tilde\eta\times\sigma}{\nu\times\sigma}-\frac12\norm{\tilde\eta\times\sigma}^2\right). \label{eqGamma}
		\end{align}
		Here, we denote by $\pi_\sigma$ the orthogonal projection onto $\sigma^\perp$ as in \autoref{eq:projection}.
	\end{lemma}
	
	\begin{proof}
		According to \autoref{thConstantSol} we have that $\hat\sigma(-\eta+\tilde\eta,0,0)=\sigma$ and $\hat\eta(-\eta+\tilde\eta,0,0)=\tilde\eta$. Hence, we obtain from \autoref{thSigma} that
		\begin{equation}\label{eq:sigmaSpecialDer}
		\begin{aligned}
			&\partial_\varepsilon\hat\sigma(-\eta+\tilde\eta,0,0) = \frac12v_1(\tilde\eta), \\
			&\partial_\delta\hat\sigma(-\eta+\tilde\eta,0,0) = -\frac12v_1(\tilde\eta)-\inner{\tilde\eta}\nu v_2(\tilde\eta),\text{ and} \\
			&\partial_{\delta\varepsilon}\hat\sigma(-\eta+\tilde\eta,0,0) = -\frac12v_2(\tilde\eta)+\left(\frac14\norm{v_1(\tilde\eta)}^2+\frac12\inner{\tilde\eta}\nu\inner{v_1(\tilde\eta)}{v_2(\tilde\eta)}\right)v_3(\tilde\eta),
		\end{aligned}
		\end{equation}
		where $(v_j(\tilde\eta))_{j=1}^3$ is the dual basis to $\nu$, $\tilde\eta-\sigma$, and $\sigma$. The dual basis vectors are given by
		\begin{equation}\label{eq:basisv}
		\begin{aligned}
			&v_1(\tilde\eta) = \frac{(\tilde\eta-\sigma)\times\sigma}{\inner{(\tilde\eta-\sigma)\times\sigma}\nu} = \mu(\tilde\eta)\tilde\eta\times\sigma, \\
			&v_2(\tilde\eta) = \frac{\nu\times\sigma}{\inner{\nu\times\sigma}{\tilde\eta-\sigma}} = -\mu(\tilde\eta)\nu\times\sigma,\text{ and} \\
			&v_3(\tilde\eta) = \frac{\nu\times(\tilde\eta-\sigma)}{\inner{\nu\times(\tilde\eta-\sigma)}\sigma} = \mu(\tilde\eta)\nu\times(\tilde\eta-\sigma).
		\end{aligned}
		\end{equation}
		where $\mu$ is defined by \autoref{eqMu}.
	
			For the derivatives with respect to $\varepsilon$ and $\delta$, this directly yields \autoref{eqSigmaEps} and \autoref{eqSigmaDelta}.
			
			To write the vector $\pi_\sigma(v_3(\tilde\eta))$ in the basis $\tilde\eta\times\sigma$ and $\nu\times\sigma$ of $\mathrm T_\sigma\S^2_{k_0}$, we consider the dual basis
			\begin{align*}
				&u_1(\tilde\eta)\coloneqq\frac{(\nu\times\sigma)\times\sigma}{\inner{(\nu\times\sigma)\times\sigma}{\tilde\eta\times\sigma}} = -\frac{\mu(\tilde\eta)}{k_0^2}(\nu\times\sigma)\times\sigma, \\
				&u_2(\tilde\eta)\coloneqq\frac{(\tilde\eta\times\sigma)\times\sigma}{\inner{(\tilde\eta\times\sigma)\times\sigma}{\nu\times\sigma}} = \frac{\mu(\tilde\eta)}{k_0^2}(\tilde\eta\times\sigma)\times\sigma.
			\end{align*}
			
			Next, we compute the projections of $\nu\times\tilde{\eta}$ onto the dual basis
			\[ \inner{u_1(\tilde\eta)}{\nu\times\tilde\eta} = -\frac{\mu(\tilde\eta)}{k_0^2}\inner{(\nu\times\sigma)\times\sigma}{\nu\times\tilde\eta} = -\frac{\mu(\tilde\eta)}{k_0^2}\big(\inner{\nu\times\sigma}\nu\inner\sigma{\tilde\eta}-\inner{\nu\times\sigma}{\tilde\eta}\inner\sigma\nu\big) = -\frac1{k_0^2}\inner\sigma\nu \]
			and
			\[ \inner{u_2(\tilde\eta)}{\nu\times\tilde\eta} = \frac{\mu(\tilde\eta)}{k_0^2}\inner{(\tilde\eta\times\sigma)\times\sigma}{\nu\times\tilde\eta} = \frac{\mu(\tilde\eta)}{k_0^2}\big(\inner{\tilde\eta\times\sigma}\nu\inner\sigma{\tilde\eta}-\inner{\tilde\eta\times\sigma}{\tilde\eta}\inner\sigma\nu\big) = \frac1{k_0^2}\inner\sigma{\tilde\eta}, \]
			Using these projections, we can now write $\pi_\sigma(v_3(\tilde\eta))$ as a linear combination of the basis vectors $\tilde\eta\times\sigma$ and $\nu\times\sigma$ by expressing $\nu\times\tilde\eta$ in the last equation of \autoref{eq:basisv} in this basis:
			\begin{align*}
				\pi_\sigma(v_3(\tilde\eta)) 
				&=			-\mu(\tilde\eta)\nu\times\sigma+\mu(\tilde\eta)\inner{u_1(\tilde\eta)}{\nu\times\tilde\eta}\tilde\eta\times\sigma+\mu(\tilde\eta)\inner{u_2(\tilde\eta)}{\nu\times\tilde\eta}\nu\times\sigma \\
				&= \frac{\mu(\tilde\eta)}{k_0^2}(\inner\sigma{\tilde\eta}-k_0^2)\nu\times\sigma-\frac{\mu(\tilde\eta)}{k_0^2}\inner\sigma\nu\tilde\eta\times\sigma.
			\end{align*}
			
			Plugging this into the last equation of \autoref{eq:sigmaSpecialDer}, we get with
			\[ \inner{v_1(\tilde\eta)}{v_2(\tilde\eta)}=-\mu^2(\tilde\eta)\inner{\tilde\eta\times\sigma}{\nu\times\sigma} \]
			that
			
			\begin{align*}
				&\pi_\sigma\big(\partial_{\delta\varepsilon}\hat\sigma(-\eta+\tilde\eta,0,0)\big) = \frac12\mu(\tilde\eta)\nu\times\sigma \\
				&\qquad\qquad+\frac{\mu^3(\tilde\eta)}{k_0^2}\left(\frac14\norm{\tilde\eta\times\sigma}^2-\frac12\inner{\tilde\eta}\nu\inner{\tilde\eta\times\sigma}{\nu\times\sigma}\right)\left((\inner\sigma{\tilde\eta}-k_0^2)\nu\times\sigma-\inner\sigma\nu\tilde\eta\times\sigma\right) \\
				&\qquad= \frac12\mu^3(\tilde\eta)\big(\alpha(\tilde\eta)\tilde\eta\times\sigma+\beta(\tilde\eta)\nu\times\sigma\big).
			\end{align*}
	\end{proof}

	\begin{lemma}\label{thCoeff}
		Writing $\tilde\eta\in S_{e_3}$ in the orthogonal basis $\sigma$, $\frac{\nu\times\sigma}{\norm{\nu\times\sigma}^2}$, $\frac{\sigma\times(\nu\times\sigma)}{\norm{\nu\times\sigma}^2}$, that is,
		\begin{equation}\label{eq:etaBasis} 
		\tilde\eta = \tilde\eta_1\sigma+\tilde\eta_2\frac{\nu\times\sigma}{\norm{\nu\times\sigma}^2}+\tilde\eta_3\frac{\sigma\times(\nu\times\sigma)}{\norm{\nu\times\sigma}^2}, 
		\end{equation}
		we get for the functions $\mu$, $\alpha$, $\beta$, and $\gamma$ defined in \autoref{eqMu}, \autoref{eqAlpha}, \autoref{eqBeta}, and \autoref{eqGamma} the expressions
		\begin{align*}
			&\mu(\tilde\eta) = -\frac1{\tilde\eta_2}, \\
			&\alpha(\tilde\eta) = \inner\sigma\nu\gamma(\tilde\eta), \\
			&\beta(\tilde\eta) = \tilde\eta_2^2-k_0^2(\tilde\eta_1-1)\gamma(\tilde\eta),\text{ and} \\
			&\gamma(\tilde\eta) = \inner\sigma\nu\tilde\eta_1\tilde\eta_3-\frac{\tilde\eta_2^2}{2\norm{\nu\times\sigma}^2}+\left(1-\frac{k_0^2}{2\norm{\nu\times\sigma}^2}\right)\tilde\eta_3^2.
		\end{align*}
	\end{lemma}
	
	\begin{proof}
		Writing $\tilde{\eta}$ as in \autoref{eq:etaBasis} gives us
		\[ \frac1{\mu(\tilde\eta)} = \inner\nu{\tilde\eta\times\sigma} = \frac{\tilde\eta_2}{\norm{\nu\times\sigma}^2}\inner\nu{(\nu\times\sigma)\times\sigma}+\frac{k_0^2\tilde\eta_3}{\norm{\nu\times\sigma}^2}\inner\nu{\nu\times\sigma} = -\tilde{\eta}_2\frac{\norm{\nu\times\sigma}^2}{\norm{\nu\times\sigma}^2}=-\tilde\eta_2. \]
		
		Moreover, we get with
		\begin{equation}\label{eqEtaNu}
			\inner{\tilde\eta}\nu = \left(\tilde\eta_1\inner\sigma\nu+\tilde\eta_3\frac{\inner{\sigma\times(\nu\times\sigma)}\nu}{\norm{\nu\times\sigma}^2}\right) = \tilde\eta_1\inner\sigma\nu+\tilde\eta_3
		\end{equation}
		that
		\begin{align*}
			\gamma(\tilde\eta) &= \frac1{k_0^2}\left(\inner{\tilde\eta}\nu\inner{\tilde\eta}{\sigma\times(\nu\times\sigma)}-\frac12\norm{\tilde\eta\times\sigma}^2\right) \\
			&= \frac1{k_0^2}\left((\tilde\eta_1\inner\sigma\nu+\tilde\eta_3)k_0^2\tilde\eta_3-\frac1{2\norm{\nu\times\sigma}^4}\norm{\tilde\eta_2(\nu\times\sigma)\times\sigma+k_0^2\tilde\eta_3\nu\times\sigma}^2\right) \\
			&= \inner\sigma\nu\tilde\eta_1\tilde\eta_3-\frac{\tilde\eta_2^2}{2\norm{\nu\times\sigma}^2}+\left(1-\frac{k_0^2}{2\norm{\nu\times\sigma}^2}\right)\tilde\eta_3^2.
		\end{align*}
		
		The expression for $\alpha$ then follows immediately and the expression for $\beta$ is obtained by remarking that
		\[ \inner\sigma{\tilde\eta}-k_0^2 = k_0^2(\tilde\eta_1-1). \]
	\end{proof}

	\FloatBarrier
	\printbibliography
	
\end{document}